\documentclass[12pt]{article}

\usepackage{amsfonts}
\usepackage{amssymb}
\usepackage{amsmath}
\usepackage{amsthm}
\usepackage{graphicx}
\usepackage{eucal}
\usepackage{ wasysym }
\usepackage{tikz}
\usepackage{tkz-berge} 
\usetikzlibrary{arrows,decorations.markings}
\usepackage{hyperref}
\hypersetup{colorlinks=true}
\usepackage{verbatim}
\usepackage{todonotes}
\usepackage{adjustbox}
\usepackage{multicol}

\theoremstyle{definition}
\newtheorem{lem}{Lemma}[section]
\newtheorem{prop}[lem]{Proposition}
\newtheorem{thm}[lem]{Theorem}
\newtheorem{cor}[lem]{Corollary}

\newtheorem{exa}[lem]{Example}
\newtheorem{cnj}[lem]{Conjecture}
\newtheorem*{Def}{Definition}

\newcommand{\bc}[2]{\genfrac{(}{)}{0pt}{}{#1}{#2}}

\newcommand{\Z}{\mathbb{Z}}

\begin{document}

\title{An Extremal Problem for the Neighborhood Lights Out Game}

\author{Lauren Keough\thanks{Department of Mathematics, Grand Valley State University, 
Allendale, Michigan
49401-6495, keoulaur@gvsu.edu,
\url{https://sites.google.com/view/laurenkeough/}}
\and Darren B. Parker\thanks{Department of Mathematics, Grand Valley State University, 
Allendale, Michigan
49401-6495, parkerda@gvsu.edu,
\url{http://faculty.gvsu.edu/parkerda}}}

\date{
\today \\
\small MR Subject Classifications: 05C57, 05C35, 05C50}

\maketitle

\begin{abstract}
Neighborhood Lights Out is a game played on graphs. Begin with a graph and a vertex labeling of the graph from the set $\{0,1,2,\dots, \ell-1\}$ for $\ell \in \mathbb{N}$. The game is played by toggling vertices: when a vertex is toggled, that vertex and each of its neighbors has its label increased by $1$ (modulo $\ell$). The game is won when every vertex has label 0.  For any $n \ge 2$ it is clear that one cannot win the game on $K_n$ unless the initial labeling assigns all vertices the same label. Given that $K_n$ has the maximum number of edges of any simple graph on $n$ vertices it is natural to ask how many edges can be in a graph so that the Neighborhood Lights Out game is winnable regardless of the initial labeling. We find the maximum number of edges a winnable $n$-vertex graph can have when at least one of $n$ and $\ell$ is odd. When $n$ and $\ell$ are both even we find the maximum size in two additional cases. The proofs of our results require us to introduce a new version of the Lights Out game that can be played given any square matrix.
\end{abstract}

\section{Introduction}

The Lights Out game was originally created by Tiger Electronics.  It has since been reimagined as a light-switching game on graphs.  Several variations of the game have been developed (see, for example \cite{Craft/Miller/Pritikin:09} and \cite{paper14}), but all have some important elements in common.  In each game, we begin with a graph $G$ and a labeling of $V(G)$ with labels in $\mathbb{Z}_\ell$ for some $\ell \ge 2$. The vertices can be toggled so as to change the labels of some of the vertices. Finally, there is some desired labeling (usually the labeling with all labels being 0, called the \emph{zero labeling} and denoted by $\textbf{0}$) that marks the end of the game.

The most common variation of the Lights Out game is what we call the \emph{neighborhood Lights Out game}.  This is a generalization of Sutner's $\sigma^+$-game (see \cite{Sutner:90}).  Each time we toggle some $v \in V(G)$, the label of each vertex in the closed neighborhood of $v$, $N[v]$, is increased by 1 modulo $\ell$.  The game is won when the zero labeling is achieved.  This game was developed independently in \cite{paper11} and \cite{Arangala:12} and has been studied in \cite{Arangala/MacDonald/Wilson:14}, \cite{Arangala/MacDonald:14}, \cite{Hope:10}, \cite{paper13}, and \cite{Hope:19}.  The original Lights Out game is the neighborhood Lights Out game on a grid graph with $\ell=2$ and has been studied in \cite{Amin/Slater:92}, \cite{Goldwasser/Klostermeyer:07}, and \cite{Sutner:90}.

It is possible for a Lights Out game to be impossible to win.  Much of the work on Lights Out games has centered on the conditions under which winning the game is possible.  Winnability depends on the version of the game that is played, the graph on which the game is played, and on $\ell$.  In our paper, we work with the neighborhood Lights Out game with labels in $\mathbb{Z}_\ell$ for arbitrary $\ell \ge 2$.

For each $n \ge 2$ there exist many labelings of $K_n$ for which the neighborhood Lights Out game is impossible to win.  In fact, any initial labeling in which not every vertex has the same label cannot be won.  It is also true that $K_n$ has the most edges of any simple graph on $n$ vertices.  It then makes sense to ask, given $n,\ell \ge 2$, what is the maximum size of a simple graph on $n$ vertices with labels in $\mathbb{Z}_\ell$ for which the neighborhood Lights Out game can be won for every possible initial labeling?  We call this maximum size $\max(n,\ell)$.  In addition, we seek to classify the winnable graphs of maximum size among all graphs on $n$ vertices with labels from $\mathbb{Z}_{\ell}$, which we call \emph{$(n,\ell)$-extremal graphs}.

 It appears that the complements of $(n,\ell)$-extremal graphs have the property that every non-pendant vertex is adjacent to a pendant vertex. As in \cite{Graf:14} we write $H\astrosun K_1$ for the graph in which, for each vertex $v$ of $H$, we add a new vertex adjacent to only $v$.  We call such graphs \emph{pendant graphs}.  In the case that a pendant graph is a tree or a forest, we use the terms \emph{pendant tree} or \emph{pendant forest}, respectively.



Our main results are in Section~\ref{Extremal Graphs}, where we determine partial results on the classification of $(n,\ell)$-extremal graphs.  In the case of $n$ odd, we show that all $(n,\ell)$-extremal graphs are complements of near perfect matchings. We also classify all $(n,\ell)$-extremal graphs when $n$ is even and $\ell$ is odd. In the remaining case we have the following conjecture.

\begin{cnj}\label{conjecture}
For $n,\ell$ even then
    \[\max(n,\ell) = \binom{n}{2} - \left(\frac{n}{2} + k\right)\]
where $k$ is the smallest nonnegative integer such that $\gcd(n-2k-1,\ell)=1$. In each case the $(n,\ell)$-extremal graphs are precisely the complements of pendant graphs of order $n$ that have size $\binom{n}{2} - \left(\frac{n}{2} + k\right)$.
\end{cnj}

By proving that the complements of pendant graphs can be won no matter the initial labeling, we conclude that $\max(n,\ell)$ is at least the quantity given in Conjecture \ref{conjecture}. We also prove equality for $0\leq k \leq 3$ and in the family of all graphs that have minimum degree at least $n-3$.

To determine winnability, we depend heavily on linear algebra methods similar to those in \cite{Anderson/Feil:98}, \cite{Arangala/MacDonald/Wilson:14}, \cite{Hope:10}, and \cite{paper11}.  We discuss these methods in Section \ref{Linear Algebra}. Our techniques differ in that we introduce how to play Lights Out given any square matrix. These tools allow us to determine winnability in some dense graphs by considering winnability in a modified Lights Out game in their sparse complements, which we discuss further in Section \ref{Winnability}.

Throughout the paper, we assume the vertex labels of any labeling are from $\mathbb{Z}_\ell$ for some $\ell \in \mathbb{N}$, and so any reference to $\ell$ refers to this set of labels.


\section{Linear Algebra} \label{Linear Algebra}
Winnability in the Lights Out game on graphs can be studied by determining a strategy for toggling the vertices.  But it can also be determined using linear algebra.  We proceed as in \cite{Anderson/Feil:98} and \cite{paper11}.  

Let $G$ be a graph with $V(G)=\{v_1,v_2, \ldots , v_n\}$, and let $N(G)=[N_{ij}]$ be the neighborhood matrix of $G$ (where $N_{ij}=1$ if and only if $v_i$ is adjacent to $v_j$ or $i=j$ and $N_{ij}=0$ otherwise).   Define the vectors $\textbf{b}, \textbf{x} \in \mathbb{Z}_\ell^n$ so that $\textbf{b}[i]$ is the initial label of $v_i$, and $\textbf{x}[i]$ is the number of times $v_i$ is toggled.  As explained in \cite[Lemma~3.1]{paper11}, $\textbf{x}$ represents a winning set of toggles if and only if it satisfies the matrix equation $N(G)\textbf{x}=-\textbf{b}$. In this linear algebra perspective we typically think of the initial labeling of the graph as a vector (as in $\textbf{b}$ above). When we determine winnability by playing the game we will typically think of the initial labeling as a function.


As described above, the neighborhood Lights Out game can be played by knowing the neighborhood matrix and an initial labeling. However, we can also play a Lights Out game using any matrix. Let $M = [m_{ij}] \in M_n(\mathbb{Z}_\ell)$ (the set of $n\times n$ matrices with entries in $\mathbb{Z}_{\ell}$), and define the \emph{vertex set of $M$} as a set of $n$ elements $V(M) = \{ v_1,v_2,\ldots,v_n \}$.  We then define the $M$-Lights Out game as follows.  We label the elements of $V(M)$ with a vector $\textbf{b} \in \mathbb{Z}_\ell^n$, where each $v_i$ has label $\textbf{b}[i]$.  We play the game by toggling elements of $V(M)$.  Each time $v_j$ is toggled, we add $m_{ij}$ to the label of $v_i$ for all $1 \leq i \leq n$.  As with the ordinary Lights Out game, we win the game when we achieve the  labeling ${\bf 0}$.


\begin{exa}
Let $G$ be a graph.  For $M=N(G)$, we get the neighborhood Lights Out game.  If we let $M$ be the adjacency matrix $A(G)$, we get an analogue of the $\sigma$-game from Sutner (see \cite{Sutner:90}), where toggling a vertex $v$ increases the label of each vertex in the open neighborhood $N_G(v)$ of $v$ by 1 modulo $\ell$ and leaves the label of $v$ unchanged. We call this the \emph{adjacency Lights Out game}.
\end{exa}

Throughout, we shorten the names of the neighborhood Lights Out game and the adajacency Lights Out game to the $N(G)$-Lights Out game and the $A(G)$-Lights Out game, respectively. We shorten even further to the $N$-Lights Out game and the $A$-Lights Out game when the graph is clear. 
Though the adjacency matrix and the neighborhood matrix are both symmetric there is no requirement that $M$ be symmetric in the $M$-Lights Out game. Now we introduce some terminology related to whether a given $M$-Lights Out game can be won.

\begin{Def}
Let $M$ and $V=V(M)$ be as above. We call a labeling $\pi$ \emph{$M$-winnable} if the $M$-Lights Out game can be won with initial labeling $\pi$.  We say that $V$ is \emph{$M$-always winnable}, or \emph{$M$-AW} for short, if all labelings of $V(M)$ are $M$-winnable.
\end{Def}

In the case that $V$ is the vertex set of a graph $G$, we refer to $G$ as being $M$-AW, with the understanding that $V(M)=V(G)$. In these cases, $M$ is often given by the neighborhood matrix or adjacency matrix.  The following summarizes the connection between whether a given $M$-Lights Out game can be won and the linear algebraic properties of $M$.  The proof follows from basic linear algebra.

\begin{lem} \label{matrixLO}  Let $M \in M_n(\mathbb{Z}_\ell)$ and $V(M)=\{v_1,v_2,\dots,v_n\}$.  
	\begin{enumerate}
	\item \label{matrixLOsys} Let $\pi$ be a labeling of $V(M)$, and define $\textbf{b}[i] = \pi(v_i)$. Then $\pi$ is $M$-winnable with the toggles given by $\textbf{x}$ if and only if $M\textbf{x} = -\textbf{b}$.
	\item  \label{matrixLOinv} The vertex set $V(M)$ is $M$-AW if and only if $M$ is invertible over $\Z_{\ell}$. 
	\end{enumerate} \end{lem}
	
In this paper, we focus on whether or not a graph $G$ is $N(G)$-AW, so we seek to determine whether or not a given neighborhood matrix is invertible. One straightforward way to apply linear algebra techniques is when two rows or columns of a matrix are identical.

\begin{Def}
Let $M \in M_n(\mathbb{Z}_\ell)$, and let $v,w \in V(M)$.  We call $v$ and $w$ \emph{$M$-twins} if the rows or columns of $M$ represented by $v$ and $w$ are identical.
\end{Def}

In graph theory, two vertices $v$ and $w$ are twins provided that have the same open neighborhood excluding $v$ and $w$. 
Twin vertices  that are adjacent in a graph  result in identical rows and columns in the neighborhood matrix and thus are $N$-twins. Twin vertices that are not adjacent result in identical rows in the adjacency matrix and thus are $A$-twins. The following is immediate from considering the invertibility of the matrix.

\begin{cor}\label{twins}
Let $M \in M_n(\mathbb{Z}_\ell)$, and suppose there exist $M$-twins in $V(M)$.  Then $V(M)$ is not $M$-AW.
\end{cor}

Note that $\mathbb{Z}_\ell$ is generally not a field, but we can still use the determinant of a matrix to determine its invertibility.  In particular, a matrix is invertible if and only if its determinant is a unit \cite[Corollary~2.21]{Brown:93}.  As in standard linear algebra, we can apply row operations to a matrix and leave the determinant unchanged or multiplied by a unit.  In particular, the typical elementary row operations (multiplying a row by a unit in $\mathbb{Z}_{\ell}$, adding an integer multiple of one row to another, and switching two rows) have no effect on whether or not the determinant is a unit.


We say that $M$ is \emph{row equivalent} to $M'$ if and only if $M$ can be turned into $M'$ by applying a sequence of elementary row operations. Since elementary row operations do not change whether or not the determinant is a unit, if $M, M' \in M_n(\mathbb{Z}_\ell)$ such that $M$ is row equivalent to $M'$ then $M$ is invertible if and only if $M'$ is invertible. Thus, if $M$ and $M'$ are row equivalent then a common vertex set $V$ is $M$-AW if and only if $V$ is $M'$-AW. 


Thus, we can determine whether or not a set $V$ is $M$-AW by applying some elementary row operations to $M$ to obtain $M'$, and then determining whether or not $V$ is $M'$-AW. 

We now apply this strategy to the neighborhood Lights Out game.  Our general strategy is to use elementary row operations to transform $N(G)$ into a matrix whose Lights Out game is easy to play.  Our first result using this technique will be for graphs that have a dominating vertex. Given graphs $G$ and $H$ we use $G\cup H$ to denote the disjoint union of the graphs.  

\begin{thm}
\label{subsetjoindom}  Let $G$ be a graph.  Then $\overline{G \cup K_1}$ is $N$-AW if and only if $G$ is $A$-AW.
\end{thm}

\begin{proof}
We have
\[
N(\overline{G \cup K_1}) = \left[
\begin{array}{c|c}
N(\overline{G}) & 1 \\ \hline
1 & 1 \\ 
\end{array}
\right]
\]
where the last row and column represent $V(K_1)$. We multiply each row except the last by the unit $-1$ and then add to each of those rows the last row.  This turns every $1$ of $N(\overline{G})$ into a $0$ and vice versa, resulting in the adjacency matrix of $G$.  Thus, we get that $N(\overline{G \cup K_1})$ is row equivalent to
\[
M = \left[
\begin{array}{c|c}
A(G) & 0 \\ \hline
1 & 1 \\
\end{array}
\right]
.\]

Thus it suffices to show that $\overline{G \cup K_1}$ is $M$-AW if and only if $G$ is $A$-AW.  Note that the $M$-Lights Out game is played as the $A$-Lights Out game on $G$, each vertex toggled in $V(G)$ adds 1 to the label of the vertex $v \in V(K_1)$,  and toggling $v$ increases its own label by 1 and has no other effect.

First suppose that $G$ is $A$-AW, and let  $\pi$ 
be a labeling of $\overline{G\cup K_1}$.
Since $G$ is $A$-AW, we can toggle the vertices of $G$ in a way that wins the $A(G)$-Lights Out game for the labeling $\pi \mid_{V(G)}$.  At this point, every vertex has label 0 except $v$.  We then toggle $v$ until it has label 0.  In the $M$-Lights Out game, toggling $v$ has no effect on labels of other vertices, so this wins the $M$-Lights Out game. Thus $\overline{G \cup K_1}$ is $M$-AW.

Conversely, suppose that $G$ is not $A$-AW.  We then give $V(G)$ a labeling that is not $A(G)$-winnable.  In the $M$-Lights Out game the only vertices that affect the labels of $V(G)$ are the vertices in $V(G)$, so this is not a winnable labeling for the $M$-Lights Out game.  Thus, $\overline{G \cup K_1}$ is not $M$-AW, which completes the proof.
\end{proof}

\section{Winnability in Dense Graphs} \label{Winnability}

In proving Theorem~\ref{subsetjoindom}, we use elementary row operations to convert the neighborhood Lights Out game on a dense graph into something resembling the adjacency Lights Out game on a sparse graph.  Since the extremal problem we are working on seeks dense, winnable graphs and playing the game on sparse graphs is typically easier, this technique works to our advantage.  The next result  allows us to make a graph denser by removing an edge from the complement graph  when the complement graph is combined with $P_4$.  

\begin{thm} \label{joinp4}
Let $G$ be a graph, $U \subseteq V(G)$ and $v$ be an end vertex of $P_4$.  Let $H$ be the graph where $V(H)=V(G) \cup V(P_4)$ and $E(H) = E(G) \cup E(P_4) \cup \{ uv: u \in U\}$.  Then $\overline{H}$ is $N$-AW if and only if $\overline{G \cup P_4}$ is $N$-AW.
\end{thm}

\begin{proof}
Let $V=V(\overline{G \cup P_4})=V(\overline{H})$, and let $P_4$ in both $G \cup P_4$ and $H$ be given by $vv_2v_3v_4$.  Note that $\overline{P_4}$ is the path given by $v_2v_4vv_3$.  By \cite[Thm. 4.3]{paper11}, $P_4$ is $N$-AW for all $\ell$. It follows that in both $\overline{H}$ and $\overline{G \cup P_4}$, the subgraph induced by $\{ v,v_2,v_3,v_4\}$ is $N$-AW.  Thus, we can toggle the vertices of $P_4$ in such a way that each vertex in $P_4$ has label zero.

We first assume $\overline{H}$ is $N$-AW and show $\overline{G \cup P_4}$ is $N$-AW.  To that end, we let $\pi: V \rightarrow \mathbb{Z}_\ell$ and show that $\pi$ is winnable on $\overline{G \cup P_4}$.  As discussed above, we can assume that $\pi \mid_{V(P_4)} = 0$.  Since $\overline{H}$ is $N$-AW, $\pi$ is winnable on $\overline{H}$.  In this winning strategy, let each $w \in V(G)$ be toggled $x_w$ times, and let $v_2$ be toggled $x$ times.  If we apply this strategy to $\overline{H}$ but refrain from toggling $v$, $v_3$, and $v_4$, this leaves $v_2$ and $v_4$ with label $x+\sum_{w \in V(G)} x_w$, $v$ with label $\sum_{w \in V(G)-U}x_w$, and $v_3$ with label $\sum_{w \in V(G)}x_w$.  Since $v_4$ is the only remaining vertex adjacent to $v_2$, $v_4$ must be toggled $-x-\sum_{w \in V(G)}x_w$ times.  This will leave both $v_2$ and $v_4$ with label zero.  Since $v$ is the only remaining vertex adjacent to $v_4$, this means we do not toggle $v$ at all.  Thus, $v_3$ (the only remaining untoggled vertex) must make its own label zero by being toggled $-\sum_{w \in V(G)}x_w$ times.  This completes winning the game on $\overline{H}$.  An important observation is that the vertices of $P_4$ are collectively toggled $-2\sum_{w \in V(G)}x_w$ times, and none of those toggles come from $v$.  Since each of $v_2$, $v_3$, and $v_4$ is adjacent to every vertex in $V(\overline{G})$, this implies that toggling the vertices of $P_4$ adds $-2\sum_{w \in V(G)}x_w$ to the labeling of each vertex in $V(G)$.  Looked at another way, if we only toggle the vertices in $V(G)$, this leaves each such vertex with label $2\sum_{w \in V(G)}x_w$.

With the initial labeling $\pi$, we now apply the above toggling strategy to $V(G)$ in $\overline{G \cup P_4}$.  By the above, each vertex in $V(G)$ has label $2\sum_{w \in V(G)}x_w$.  Since $v$ and each of the $v_i$ are adjacent to all vertices in $V(G)$, it follows that toggling the vertices in $V(G)$ leaves $v$ and each $v_i$ with label $\sum_{w \in V(G)}x_w$.  Each of $v_2$ and $v_3$ is now toggled $-\sum_{w \in V(G)}x_w$ times.  This makes the label of $v$ and each $v_i$ zero.  In addition, it adds $-2\sum_{w \in V(G)}x_i$ to the labels of $V(G)$, which gives each of them label zero as well.

We proceed similarly for the converse.  Assume $\overline{G \cup P_4}$ is $N$-AW, and let $\pi: V \rightarrow \mathbb{Z}_\ell$ be a labeling as above with $\pi \mid_{V(P_4)}=0$.  We need to prove that $\pi$ is winnable on $\overline{H}$.   As before, there is a winning toggling strategy for $\overline{G \cup P_4}$, where each $w \in V(G)$ is toggled $x_w'$ times, and $v_2$ is toggled $x'$ times.  At this point, we determine the toggles for $v$ and each remaining $v_i$ as before, and it follows that the vertices are collectively toggled $-2\sum_{w \in V(G)}x_w'$ times.  As before, this implies that toggling the vertices of $V(G)$ results in the label of each vertex in $V(G)$ being $2\sum_{w \in V(G)}x_w'$.

Again, we apply the above toggling strategy just to the vertices of $V(G)$ in $\overline{H}$.  This leaves each of $v_2$, $v_3$, and $v_4$ with label $\sum_{w \in V(G)}x_w'$ and $v$ with label $\sum_{w \in V(G)-U}x_w'$.  We then win the game as follows: $v_2$ is toggled $-2\sum_{w \in U}x_w'-\sum_{w \in V(G)-U} x_w'$ times, $v_3$ is toggled $-\sum_{w \in V(G)}x_w'$ times, and $v_4$ is toggled $\sum_{w \in U}x_w'$ times.
\end{proof}

We can apply this result to complements of graphs that include components that are path graphs. For $k\in\mathbb{N}$ and $G$ a graph we use $kG$ to denote $k$ disjoint copies of $G$.

\begin{cor} \label{pathrestrictions}
Let $G$ be a graph of order $n$ that is $N$-AW.
\begin{enumerate}
\item \label{nopath3} No component of $\overline{G}$ can be $P_k$ such that $k$ is congruent to 3 mod 4.
\item \label{onepath4} At most one component of $\overline{G}$ can be $P_k$ such that $k$ is congruent to 1 modulo 4.
\item \label{extpathbound} If $\overline{G}$ is an $(n,\ell)$-extremal graph, then no component of $G$ is a path of order more than 4.
\end{enumerate}
\end{cor}

\begin{proof}
For (\ref{nopath3}), let $P$ be a component of $\overline{G}$ that is a path of order $4k+3$ with $k \in \mathbb{N} \cup \{0\}$.  By Lemma~\ref{joinp4}, if we replace $P$ in $\overline{G}$ with $kP_4 \cup P_3$, the complement of the resulting graph is $N$-AW if and only if $G$ is.  Thus, we can assume $P=P_3$.  However, the end vertices of the $P_3$ component in $\overline{G}$ are $N(G)$-twins in $G$, so $G$ is not $N$-AW by Corollary \ref{twins}.

For (\ref{onepath4}), we apply Lemma~\ref{joinp4} again. If we have more than one component of $\overline{G}$ is a path with order congruent to 1 modulo 4, we can assume that all such components are $P_1$.  But the vertices of these components are all $N(G)$-twins, and so in order for $G$ to be $N$-AW, $\overline{G}$ can have at most one component be a path of order congruent to 1 modulo 4.

Finally, (\ref{extpathbound}) follows from the fact that if we replace the component of $\overline{G}$ that is $P_k$ with $k > 4$ with $P_{k-4} \cup P_4$, the complement of the resulting graph will be $N$-AW with larger size, thus contradicting the assumption that $\overline{G}$ is $(n,\ell)$-extremal.
\end{proof}

Given a matrix $M$, let $\pi$ be a labeling of $V(M)$.  For $U \subseteq V(M)$ and $r \in \mathbb{Z}_\ell$, we define the labeling $\pi_{U,r}: V(M) \rightarrow \mathbb{Z}_\ell$ as 
\[
\pi_{U,r}(v) =\begin{cases} \pi(v)+r & v \in U \\ \pi(v) & v \notin U \end{cases}  .
\]
In the case $U=V(M)$, we write $\pi_{V(M),r}=\pi_r$.

%

When we encounter these labelings in the proof of Theorem~\ref{pendantremove}, we are concerned not only if certain labelings are winnable, but also how many toggles can be used to win the game for these labelings.  Recall that $\textbf{0}$ is the zero labeling, which assigns to every vertex a label of 0.

\begin{Def}
Let $M \in M_n(\mathbb{Z}_\ell)$, $r \in \mathbb{Z}_\ell$ and $U \subseteq V(M)$.  We define the set of \emph{$U$-toggling numbers} $T_U^M(r) \subseteq \mathbb{Z}_\ell$ as follows.  We say $t \in T_U^M(r)$ if the elements of $V(M)$ can be toggled to win the $M$-Lights Out game with initial labeling $\textbf{0}_{U,r}$ in such a way that the vertices in $U$ are collectively toggled $t$ times.
\end{Def}

Note that each number in $T_U^M(0)$ corresponds to a set of toggles that leaves the initial labeling unchanged. Such sets of toggles are called \emph{null toggles}.  Null toggles function very similarly to null spaces of a linear transformation.  For instance, there exist two sets of toggles with $t$ toggles and $t'$ toggles of the vertices of $U$, respectively, to have the same effect on the labels of $V(M)$ if and only if $t'=t+q$ for some $q \in T_U^M(0)$.

In both the neighborhood and adjacency Lights Out games, winning a particular game is equivalent to winning the game on each individual connected component.  This simplifies the computation of toggling numbers in these cases.  Let $G$ be a graph with $U \subseteq V(G)$ and $M$ is $N(G)$ or $A(G)$. If $G_1,G_2, \ldots,G_c$ are the connected components of $G$, and $U_i = U \cap V(G_i)$, then  $T_U^M(r) = \{ \sum_{i=1}^c t_i : t_i \in T_{U_i}^{M_i}(r) \}$, where $M_i$ is $N(G_i)$ or $A(G_i)$, respectively.

Suppose we have two different sets of toggles and look at their effect individually on each vertex. For each $v \in V(M)$, suppose that the label of $v$ is increased by $r_v$ for the first set of toggles and is increased by $s_v$ for the second set of toggles.  Then combining the two sets of toggles increases each $v \in V(M)$ by $r_v+s_v$.  We use this observation to prove the following.

\begin{lem} \label{basictoggling}
Let $n\in\mathbb{N}$ and $M \in M_n(\mathbb{Z}_\ell)$, let $U \subseteq V(M)$, and let $r \in \mathbb{N}$ be minimal such that $T_U^M(r) \neq \emptyset$.  Then $r \mid \ell$, and $T_U^M(s) \neq \emptyset$ if and only if $r \mid s$.
\end{lem}

\begin{proof}
It is easy to show that $\{ s \in \mathbb{Z}_\ell: T_U^M(s) \ne \emptyset\}$ is an additive subgroup of $\mathbb{Z}_\ell$.  The result follows easily.
\end{proof}

For graphs with a pendant vertex, it will be helpful to understand the relationship between winning the adjacency game on both the graph and a certain subgraph.

\begin{lem} \label{noU}
Let $G$ be a graph with a pendant vertex $p$. Let $v$ be the neighbor of $p$ in $G$, let $G'$ be the graph induced by $V(G)-\{p,v\}$, let $U=N_G(v)-\{p\}$. Finally, let $\pi$ be a labeling of $G$, and define the labeling $\pi'$ on $G'$ by
\[
\pi'(w) = \begin{cases} \pi(w) - \pi(p) & w \in U \\ \pi(w) & \hbox{otherwise} \end{cases}.
\]
Then
\begin{enumerate}
    \item \label{ptop'} $\pi'$ is $A(G')$-winnable with $t$ toggles from $V(G')-U$ (along with perhaps some toggles from $U$) if and only if $\pi$ is $A(G)$-winnable with $t-\pi(v)-\pi(p)$ toggles from $V(G)$.
    \item \label{togsame} If $s \in \mathbb{Z}_\ell$, then $T_{V(G)}^{A(G)}(s) = \{ t-2s: t \in T_{V(G')-U}^{A(G')}(s)\}$.
\end{enumerate}
\end{lem}

\begin{proof}
For (\ref{ptop'}), we first assume $\pi'$ is $A(G')$-winnable with $t$ toggles from $V(G')-U$.  If we begin with the labeling $\pi$ on $G$, we begin by toggling the vertices as we would to win the adjacency game on $G'$ with labeling $\pi'$.  When we do this, we subtract $\pi(w)$ from the label of each $w \in V(G')-U$ and subtract $\pi(w)-\pi(p)$ from each $w \in U$.  This leaves each vertex in $V(G')-U$ with label 0 and each vertex in $U$ with label $\pi(p)$.  If the vertices of $U$ get toggled $t_U$ times, it also leaves $v$ with label $\pi(v)+t_U$.  Then $v$ is toggled $-\pi(p)$ times and $p$ is toggled $-\pi(v)-t_U$ times to win the game.  The total number of toggles is $t+t_U-\pi(p)-\pi(v)-t_U = t-\pi(p)-\pi(v)$.

If we assume $\pi$ is $A(G)$-winnable with $t-\pi(p)-\pi(v)$ toggles, note that since $v$ is the only neighbor to $p$ in $G$, $v$ must be toggled $-\pi(p)$ times to win the $A(G)$-Lights Out game with initial labeling $\pi$.  This leaves each $w \in V(G')$ with label $\pi'(w)$.  We then toggle the vertices of $G'$ as we do for winning the adjacency game on $G$ with initial labeling $\pi$.  This will win the adjacency game on $G'$ with initial labeling $\pi'$.  Note that if $t_U$ is the number of toggles among the vertices of $U$, then that leaves $v$ with label $\pi(v)+t_U$.  This requires $p$ to be toggled $-\pi(v)-t_U$ times.  If we let $t'$ be the number of toggles among vertices in $V(G')-U$, and if we total up the number of toggles altogether, we get $t-\pi(p)-\pi(v) = -\pi(p)+t'+t_U-\pi(v)-t_U = t'-\pi(p)-\pi(v)$.  Thus, $t=t'$, which proves the result.  Part (\ref{togsame}) follows from letting $\pi(w)=s$ for all $w \in V(G)$.
\end{proof}

\begin{thm} \label{pendantremove}
Let $G$ be a graph with a pendant vertex $p$.  Let $v$ be the neighbor of $p$ in $G$, and let $G'$ be the graph induced by $V(G)-\{p,v\}$.
\begin{enumerate}
\item \label{dompen} $G$ is $A(G)$-AW if and only if $G'$ is $A(G')$-AW.
\item \label{penneighbor} Let $r \in \mathbb{N}$ be minimum such that $T_{V(G)}^{A(G)}(r) \ne \emptyset$, and let $t \in T_{V(G)}^{A(G)}(r)$.  Then $\overline{G}$ is $N(\overline{G})$-AW if and only if
	\begin{enumerate}
	\item \label{allswinnable} For each labeling $\pi$ of $V(G)$, there is some  $s \in \mathbb{Z}_\ell$ such that $\pi_s$ is $A(G)$-winnable.
	\item \label{nullshit} For each $z \in \mathbb{Z}_\ell$, there exists $q \in T_{V(G)}^{A(G)}(0)$ such that there is a solution to $(r+t)x \equiv z+q$ (mod $\ell$).	\end{enumerate}
\end{enumerate}
\end{thm}

\begin{proof}
For (\ref{dompen}), we first assume $G$ is $A(G)$-AW.  Let $G'$ have an arbitrary labeling.  We extend this labeling to a labeling of $G$ by giving each of $p$ and $v$ a label of 0.  This labeling of $G$ is winnable since $G$ is $A(G)$-AW, so we toggle the vertices of $G'$ as we would in a winning toggling of $G$.  If not every label of $G'$ becomes 0, then we have to toggle $v$ to give $G$ a zero labeling.  However, this leaves $p$ with a nonzero label.  Since $v$ is the only neighbor of $p$, this implies that toggling $v$ makes the zero labeling on $G$ impossible.  Thus, the toggles we did for $G'$ leaves all vertices in $G'$ with label 0, and so $G'$ is $A(G')$-AW.

If we assume $G'$ is $A(G')$-AW and let $G$ have an arbitrary labeling, we first toggle $v$ so that $p$ has label 0.  The resulting labeling restricted to $G'$ is winnable since $G'$ is $A(G')$-AW.  We can then toggle the vertices of $G'$ so that all vertices of $G'$ have label 0.  This leaves all vertices with label 0, except perhaps $v$ since $v$ is the only vertex not in $G'$ that is adjacent to a vertex in $G'$.  We then toggle $p$ until $v$ has label 0, which wins that game.  Thus, $G$ is $A(G)$-AW.

For (\ref{penneighbor}), let $U=N_G(v)-\{p\}$.  Then $N(\overline{G})$ looks like the following.
\[
\begin{array}{c|c|c|c|c|}
& V(G')-U & U & v & p \\ \hline
V(G')-U & N(\overline{G'-U}) & * & 1 & 1 \\ \hline
U & * & N(\overline{U}) & 0 & 1 \\ \hline
v & 1 & 0 & 1 & 0 \\ \hline
p & 1 & 1 & 0 & 1 \\ \hline
\end{array}
\]
where $G'-U$ is the induced subgraph with vertex set $V(G')-U$ and the $*$ blocks are the entries that make the four top-left blocks $N(\overline{G'})$.  We multiply each row except the last by the unit $-1$ and then add to each of those rows the last row to get
\[
M=\begin{array}{c|c|c|c|c|}
& V(G')-U & U & v & p \\ \hline
V(G')-U & A(G'-U) & \overline{*} & -1 & 0 \\ \hline
U & \overline{*} & A(U) & 0 & 0 \\ \hline
v & 0 & 1 & -1 & 1 \\ \hline
p & 1 & 1 & 0 & 1 \\ \hline
\end{array}
\]
where the $\overline{*}$ blocks are obtained from $*$ by changing the 1 entries to 0 and the 0 entries to 1.  
This makes the top-left four blocks $A(G')$.  So the $M$-Lights Out game is played as the $A(G')$-Lights Out game on $V(G')$; toggling any vertex in $V(G')$ adds 1 to the label of $p$; toggling any vertex in $U$ adds 1 to the label of $v$; toggling $v$ adds $-1$ to every vertex in $(V(G)-U) \cup \{v\}$; and toggling $p$ adds 1 to $v$ and $p$.

We first assume $\overline{G}$ is $N(\overline{G})$-AW.  Since $M$ is row equivalent to $N(\overline{G})$, $\overline{G}$ is $M$-AW.  To prove (\ref{allswinnable}), consider the labeling giving $v$ and $p$ labels of 0, each $w \in U$ a label of $\pi(w)-\pi(p)$, and each $w \in V(G')-U$ a label of $\pi(w)$, which is $M$-winnable by assumption.  If we toggle $v$ and $p$ as part of a winning toggling, we get the following labeling of $V(G')$.
\[
\lambda(w) = \left\{ \begin{matrix} \pi(w)-\pi(p) & w \in U \\ \pi(w)+s & \hbox{otherwise} \end{matrix} \right.
\]
where $v$ is toggled $-s$ times.  We claim that $\pi_s$ is $A(G)$-winnable.  If we define $\pi_s'$ similarly as $\pi'$ in Lemma~\ref{noU}, then $\pi_s' = \lambda$, which we showed to be $A(G')$-winnable.  By Lemma~\ref{noU}(\ref{ptop'}), $\pi_s$ is winnable.

For (\ref{nullshit}), let $z \in \mathbb{Z}_\ell$, and consider the labeling where $p$ has label $-z$ and all other labels are 0.  This labeling is $M$-winnable by assumption, so let $y_1$ be the number of times $v$ is toggled and $y_2$ be the number of times $p$ is toggled in order to win the $M$-Lights Out game with this labeling.  This results in each vertex of $V(G')-U$ having label $-y_1$, each vertex of $U$ having label 0, $v$ having label $y_2-y_1$, and $p$ having label $-z+y_2$.

At this point, we have only the vertices in $V(G')$ to toggle, which means the remaining toggles necessary to win the $M$-Lights Out game will also win the $A(G')$-Lights Out game with labeling $\textbf{0}_{V(G')-U,-y_1}$.  Thus, $T_{V(G')-U}^{A(G')}(-y_1) \ne \emptyset$.  By Lemma~\ref{noU}(\ref{togsame}), $T_{V(G)}^{A(G)}(-y_1) \ne \emptyset$, and so $-y_1=rx$ for some $x \in \mathbb{Z}$ by Lemma~\ref{basictoggling}.  By assumption, $t \in T_{V(G)}^{A(G)}(r)$, and so $t=t'-2r$ for some $t' \in T_{V(G')-U}^{A(G')}(r)$ by Lemma~\ref{noU}(\ref{togsame}).
Thus, there exists $t_U \in \mathbb{Z}$ such that we can collectively toggle the vertices of $U$ $t_U$ times and the vertices of $V(G')-U$ $t'$ times to win the $A(G')$-Lights Out Game with labeling $\textbf{0}_{V(G')-U,r}$.  By repeating $x$ times the toggles we use for the labeling $\textbf{0}_{V(G')-U,r}$, we can toggle the vertices of $U$ and $V(G')-U$ $xt_U$ and $xt'$ times, respectively, to win the $A(G')$-Lights Out Game with labeling $\textbf{0}_{V(G')-U,rx}$.  Since toggling the vertices of $G'$ to win the $M$-Lights Out game also must win the adjacency game on $G'$ with labeling $\textbf{0}_{V(G')-U,rx}$, we must toggle the vertices of $G'$ $xt_U+xt'+k$ for some $k \in T_{V(G')}^{A(G')}(0)$.  If we let $k=q_1+q_2$, where $q_1$ is the number of toggles from $V(G')-U$ and $q_2$ is the number of toggles from $U$ in the null toggle, we have $q_1 \in T_{V(G')-U}^{A(G')}(0)$.  Note that by negating all toggles in this null toggle, we still get a null toggle, and so $-q_1 \in T_{V(G')-U}^{A(G')}(0)$.  This leaves all vertices in $V(G')$ with label 0, $v$ with label $y_2+xr+xt_U+q_2$ (by setting $-y_1=xr$), and $p$ with label $-z+y_2+xt_U+xt'+q_1+q_2$.

All of the toggles have been accounted for, and so the labels of $v$ and $p$ must be 0.  We then eliminate $y_2$ in the resulting system of equations to get $(r-t')x=-z+q_1$.  Recall that $t=t'-2r$, and so $t'=t+2r$.  This gives us $(-r-t)x=-z+q_1$, and so $(r+t)x=z-q_1$.  Now let $q=-q_1$.  As noted above, $q \in T_{V(G')-U}^{A(G')}(0)$.  By Lemma~\ref{noU}(\ref{togsame}), $T_{V(G')-U}^{A(G')}(0)=T_{V(G)}^{A(G)}(0)$, and so $q \in T_{V(G)}^{A(G)}(0)$.  Since $(r+t)x=z+q$, this proves (\ref{nullshit}).

Now we assume that (\ref{allswinnable}) and (\ref{nullshit}) hold, and we prove that $\overline{G}$ is $N(\overline{G})$-AW.  Since $M$ is row equivalent to $N(\overline{G})$, we need only prove that $\overline{G}$ is $M$-AW.  Let $\pi$ be a labeling of $V(\overline{G})$.  Consider the labeling $\lambda$ of $V(G)$ that is 0 on $p$ and $v$ and $\pi$ on $V(G')$.  By (\ref{allswinnable}), $\lambda_s$ is $(A(G)$-winnable for some $s \in \mathbb{Z}_\ell$.  If we define $\lambda_s'$ as in Lemma~\ref{noU}(\ref{ptop'}), we get $\lambda_s'=\pi_{V(G')-U,s} |_{V(G')}$.  By Lemma~\ref{noU}(\ref{ptop'}), $\pi_{V(G')-U,s} |_{V(G')}$ is $A(G')$-winnable. Then $v$ can be toggled in the $M$-Lights Out game $-s$ times to obtain $\pi_{V(G')-U,s} |_{V(G')}$ on $V(G')$, and we can then toggle the vertices of $V(G')$ to give every vertex in $V(G')$ a label of 0.  This leaves $v$ with label $a$ and $p$ with label $b$ for some $a,b \in \mathbb{Z}_\ell$.  

By Lemma~\ref{noU}(\ref{togsame}), $t=t'-2r$ for some $t' \in T_{V(G')-U}^{A(G')}(r)$. Thus, there exists $t_U \in \mathbb{Z}_\ell$ such that we can toggle the vertices of $V(G')-U$ $t'$ times and the vertices of $U$ $t_U$ times to win the $A(G')$-Lights Out game with labeling $\textbf{0}_{V(G')-U,r}$.  Lemma~\ref{noU}(\ref{togsame}) implies that $T_{V(G)}^{A(G)}(0)=T_{V(G')-U}^{A(G')}(0)$, and so $q \in T_{V(G')-U}^{A(G')}(0)$.  As we reasoned above, $-q \in T_{V(G')-U}^{A(G')}(0)$, and so there exists $q_U \in T_U^A(0)$ such that $q_U-q \in T_{V(G')}^{A(G')}(0)$.  Now let $x$ be a solution to (\ref{nullshit}), where $z=a-b$.  This gives us $(r-t')x=b-a-q$.  If $v$ is toggled $-xr$ times and $p$ is toggled $-b-x(t_U+t')+(q-q_U)$ times, this leaves each vertex of $V(G')-U$ with label $xr$, each vertex of $U$ with label 0, $v$ with label $a-b+xr-x(t_U+t')+(q-q_U)$, and $p$ with label $-x(t_U+t')+(q-q_U)$.  As we reasoned above, we can then win the $A(G')$-Lights Out game with labeling $\textbf{0}_{V(G')-U,xr}$ (and thus make the labels of $V(G')$ to be 0) by toggling the vertices of $U$ $xt_U$ times and the vertices of $V(G')-U$ a total of $xt'$ times.  The vertices of $V(G')$ can be toggled in such a way that the vertices of $U$ are toggled $q_U$ times, the vertices of $V(G')-U$ are toggled $-q$ times, and these toggles collectively have no effect on the labels of $V(G')$.  So we combine these to toggle the vertices of $U$ collectively $xt_U+q_U$ times and the vertices of $V(G')-U$ collectively $xt'-q$ times.  This leaves the vertices of $V(G')$ with label 0, $p$ with label $(-x[t_U+t']+[q-q_U])+(xt_U+q_U+xt'-q)=0$, and $v$ with label 
\begin{align*}
a-b+xr-x(t_U+t')+(q-q_U)+xt_U+q_U &= a-b+x(r-t')+q \\ &= a-b+(b-a-q)+q=0
\end{align*}
This wins the game and shows that $\overline{G}$ is $N(\overline{G})$-AW.
\end{proof}

In Theorem~\ref{pendantremove}(\ref{penneighbor}), if $G$ is $A$-AW, then $\pi_s$ is automatically $A(G)$-winnable for all $s \in \mathbb{Z}_\ell$.  Thus, $G$ satisfies Theorem~\ref{pendantremove}(\ref{allswinnable}) and makes $r=1$.  Furthermore, $A(G)$ is invertible, so the only null toggle possible is where no buttons are pushed, making $T_{V(G)}^{A(G)}(0)=\{0\}$.  This gives us the following.

\begin{cor}  \label{subsetjoinaw} Let $G$ be an $A$-AW graph with a pendant vertex. Let $t \in T_{V(G)}^{A(G)}(1)$.  Then $\overline{G}$ is $N$-AW if and only if $\gcd(1+t,\ell)=1$. \end{cor}

\begin{proof}
Since $G$ is $A$-AW, part (\ref{allswinnable}) of Theorem~\ref{pendantremove} is automatically satisfied.  Moreover, since $G$ is $A$-AW, $A$ is invertible, which implies that $T_{V(G)}^{A(G)}(0) = \{0\}$.  The result then follows directly from Theorem~\ref{pendantremove}.
\end{proof}

Furthermore, for possible $(n,\ell)$-extremal graphs with a dominating vertex, Theorem~\ref{pendantremove}(\ref{dompen}) gives us a way to eliminate most graphs with pendant vertices.

\begin{cor} \label{extdom}
Let $G$ be a graph with a dominating vertex.  If $\overline{G}$ has a pendant vertex that is not part of a component isomorphic to $P_2$, then $G$ is not $(n,\ell)$-extremal for any $n$ and $\ell$.
\end{cor}

\begin{proof}
For contradiction, assume $G$ is $(n,\ell)$-extremal, that $\overline{G}$ has a pendant vertex $p$ with neighbor $v$, and that $p$ and $v$ are not the only vertices in their connected component of $\overline{G}$.  Thus, $v$ has a neighbor other than $p$ in $\overline{G}$.  Let $w$ be the dominating vertex in $G$, and let $G'$ be the subgraph of $\overline{G}$ induced by $V(\overline{G})-\{p,v,w\}$.  If we remove the edges in $\overline{G}$ incident to $v$ but not $p$, we get the graph $H=G' \cup P_2 \cup P_1$.  Note that $H$ has size smaller than $\overline{G}$, and so $\overline{H}$ has size greater than $G$.  To contradict the assumption that $G$ is $(n,\ell)$-extremal, it then suffices to prove that $\overline{H}$ is $N$-AW.

Since $G$ is $N$-AW, Theorem~\ref{subsetjoindom} implies that $\overline{G}-\{w\}$ is $A$-AW.  By Theorem~\ref{pendantremove}(\ref{dompen}), this implies that $G'$ is $A$-AW.  Since $P_2$ is $A$-AW for all $\ell$, it follows that $G' \cup P_2 = H-\{w\}$ is $A$-AW.  By Theorem~\ref{subsetjoindom}, $\overline{H}$ is $N$-AW. This means $G$ is not $(n,\ell)$-extremal, a contradiction.
\end{proof}

One nice property of pendant graphs is that it is really easy to play the $A$-Lights Out game on them. This is demonstrated in the following result.
Recall that  $H\astrosun K_1$ denotes the graph in which, for each vertex $v$ of $H$, we add a new vertex adjacent to only $v$.
\begin{lem} \label{pendwin} Let $H$ be a graph, and $G=H \astrosun K_1$.  Then
	\begin{enumerate}
	\item \label{pendwinall} $G$ is $A$-AW for all $\ell \in \mathbb{N}$
	\item \label{pendtoggen} If $G$ has size $m$ and order $n$, then $T_{V(G)}^{A(G)}(1)=\{ 2(m-n) \}$.
	\item \label{pendtreetog} If $G$ is a pendant forest with $c$ components, then $T_{V(G)}^{A(G)}(1)=\{ -2c \}$.
	\end{enumerate}
\end{lem}

\begin{proof}
For (\ref{pendwinall}), we have the following algorithm for winning any $A$-Lights Out game on $G$.  Toggle each vertex in $V(H)$ until its pendant neighbor has label 0.  Then toggle each vertex not in $V(H)$ until its neighbor has label 0.  This results in the zero labeling, which makes $G$ $A$-AW.

For (\ref{pendtoggen}), we begin with the labeling $\textbf{0}_{1}$.  Applying the above strategy, each vertex in $V(H)$ is toggled $-1$ times, giving a total of $-\frac{n}{2}$ toggles.  Each vertex toggled also decreases by 1 the label of each adjacent vertex in $H$.  Collectively, this decreases the labels of $V(H)$ by $2|E(H)| =2 \left( m- \frac{n}{2} \right)$.  That means that when we toggle the pendant vertices, we must toggle $-1$ each for the initial label of 1 for each vertex in $H$ plus $\left( m- \frac{n}{2} \right)$ for the decrease in labels from toggling $V(H)$.  In total, we get $-\frac{n}{2} - \frac{n}{2} + 2 \left( m- \frac{n}{2} \right) = 2(m-n)$, which completes the proof.

Finally, (\ref{pendtreetog}) follows from the fact that for a forest, we have $m=n-c$.
\end{proof}

%
%

We can now determine the $N$-winnability of the complements of pendant graphs.
Interestingly, the issue of whether or not a pendant graph is $N$-AW depends entirely on the size and order of the pendant graph.

\begin{lem} \label{pendantN} Let $G$ be a pendant graph of size $m$ and order $n$.  Then $\overline{G}$ is $N$-AW if and only if $\gcd(2[n-m]-1,\ell)=1$. Equivalently, if $G$ is a graph of even order $n$ and size $\binom{n}{2} - \left(\frac{n}{2} +k\right)$ such that $\overline{G}$ is a pendant graph then $G$ is $N$-AW if and only if $\gcd(n-2k-1,\ell)=1$.
\end{lem}

\begin{proof}
By Lemma~\ref{pendwin}(\ref{pendwinall}), $G$ is $A$-AW, and so we can apply Corollary~\ref{subsetjoinaw}.  By Lemma~\ref{pendwin}(\ref{pendtoggen}), $T_{V(G)}^{A(G)}(1) = \{ 2(m-n)\}$.  By Corollary~\ref{subsetjoinaw}, $\overline{G}$ is $N$-AW if and only if $\gcd(2(m-n)+1,\ell)=1$.  The second part follows from substituting $m=\frac{n}{2}+k$ to get $2(m-n)+1=-(n-2k-1)$.
\end{proof}

If $\overline{G}$ is a forest, then $n-m$ is the number of components of $\overline{G}$.  This along with Lemma~\ref{pendantN} gives us the following.

\begin{cor} \label{2c-1} Let $G$ be a graph such that the components of $\overline{G}$ are all pendant trees.  If  $c$ is the number of components of $\overline{G}$, then $G$ is $N$-AW if and only if $\gcd(2c-1,\ell)=1$. \end{cor}

When we are classifying $(n,\ell)$-extremal graphs in Section \ref{Extremal Graphs}, it will be helpful to replace connected components of the complement of one graph with another graph without affecting the $N$-winnability of the original graph.  The following guarantees that the conditions of Theorem~\ref{pendantremove}(\ref{penneighbor}) are unaffected by the replacement.

%

\begin{cor} \label{extswitch}
Let $G$ be a graph with a pendant vertex, and let $C$ be a connected component of $G$ that is $A$-AW.  If there exists a graph $C'$ such that
\begin{enumerate}
    \item $C'$ is $A$-AW.
    \item $T_{V(C')}^{A(C')}(1)=T_{V(C)}^{A(C)}(1)$
    \item $C$ and $C'$ have the same order.
    \item $C'$ has smaller size than $C$.
\end{enumerate}
Then $\overline{G}$ is not $(n,\ell)$-extremal.
\end{cor}

\begin{proof}
Let $G'$ be the graph identical to $G$ except that the component $C$ is replaced with $C'$.  The winnability of the adjacency game is determined by the winnability of the adjacency game on each connected component of a given graph. Since both $C$ and $C'$ are $A$-AW, then given a labeling $\pi$ on $G$ (resp. a labeling $\pi'$ on $G'$), we have that $\pi_s$ is $A(G)$-winnable (resp. $\pi'_s$ is $A(G')$-winnable) if and only if $\pi_s$ restricted to $G-C$ (resp. $\pi'_s$ restricted to $G'-C'$) is $A(G-C)$-winnable (resp. $A(G'-C')$-winnable).  Since $G-C=G'-C'$, this condition is identical for both $G$ and $G'$.  Thus, either both of $G$ and $G'$ satisfy Theorem~\ref{pendantremove}(\ref{allswinnable}) or neither does.  A similar argument gives us that either both of $G$ and $G'$ satisfy Theorem~\ref{pendantremove}(\ref{nullshit}) or neither does.  Thus, $\overline{G}$ is $N$-AW if and only if $\overline{G'}$ is $N$-AW.  Furthermore, since $C$ and $C'$ have the same order, so do $\overline{G}$ and $\overline{G'}$.  Finally, since $C'$ has smaller size than $C$, $\overline{G'}$ has larger size than $\overline{G}$.  Since $\overline{G'}$ and $\overline{G}$ have the same order and same winnability but $\overline{G'}$ has larger size, $\overline{G}$ cannot be $(n,\ell)$-extremal.
\end{proof}

\section{Extremal Graphs} \label{Extremal Graphs}

Recall that $\max(n,\ell)$ is the maximum number of edges in an $N$-AW graph with $n$ vertices and a graph is $(n,\ell)$-extremal provided that it has order $n$, size $\max(n,\ell)$, and  is $N$-AW. We begin with straightforward upper and lower bounds on $\max(n,\ell)$.
 \begin{prop} \label{bestcase}
For any $n,\ell\in\mathbb{N}$ we have
	\[ \binom{n}{2} - (n-1) \leq \max(n,\ell) \leq \binom{n}{2} - \left\lfloor \dfrac{n}{2} \right\rfloor .\]
\end{prop}

\begin{proof}
For the right inequality, if $|E(\overline{G})| < \left\lfloor \frac{n}{2} \right\rfloor$, at most $\left\lfloor \frac{n}{2} \right\rfloor-1$ edges are removed from $K_n$ to obtain $G$.  Thus, at most $2\left(\left\lfloor \frac{n}{2} \right\rfloor -1\right)\le n-2$ vertices of $K_n$ can have their degrees reduced by one or more to obtain $G$.  So, at least two vertices in $G$ are dominating vertices. Two dominating vertices are $N$-twins, and such a $G$ is not $N$-AW by Corollary \ref{twins}. On the other hand we know $\max(n,\ell)\geq \binom{n}{2} - (n-1)$ since the complement of any pendant tree is $N$-AW for all $\ell$ by Corollary \ref{2c-1}.
\end{proof}

To obtain the upper bound of Proposition~\ref{bestcase}, we need $G$ to be a perfect or near-perfect matching.  Let $M_n$ be a perfect matching on $n$ vertices when $n$ is even and a near-perfect matching on $n$ vertices when $n$ is odd.  The following two results show us when $M_n$ is $(n,\ell)$-extremal.

\begin{prop}\label{prop:oddmatch}
If $n$ is odd, then \[\max(n,\ell)=\binom{n}{2}-\left\lfloor\frac{n}{2}\right\rfloor\]
for all $\ell\in\mathbb{N}$. Moreover, $\overline{M_n}$ is the unique $N$-AW of maximum size on $n$ vertices.
\end{prop}

\begin{proof}
We have that $M_{n-1}$ is a pendant graph, specifically $\left(\frac{n-1}{2}K_1\right) \astrosun K_1$, so by Lemma \ref{pendwin}(\ref{pendwinall}), $M_{n-1}$ is $A$-AW for all $\ell$. By Theorem~\ref{subsetjoindom}, $\overline{M_n}$ is $N$-AW. So, $\max(n,\ell)\geq \binom{n}{2}-\left\lfloor\frac{n}{2}\right\rfloor$ when $n$ is odd. By Proposition \ref{bestcase}, $\max(n,\ell)=\binom{n}{2}-\left\lfloor\frac{n}{2}\right\rfloor$.  
Any graph $G\neq \overline{M_n}$ with $\binom{n}{2}-\left\lfloor\frac{n}{2}\right\rfloor$ edges  must have at least two dominating vertices in $G$, which are $N$-twins. Thus, $\overline{M_n}$ is unique.
\end{proof}

However, when $n$ is even, not all complements of perfect matchings give us $(n,\ell)$-extremal graphs.

\begin{prop} \label{evenmatch}
If $n$ is even, then 
\begin{center}
$\max(n,\ell)=\binom{n}{2}-\frac{n}{2}$ if and only if $\gcd(n-1,\ell)=1$. 
\end{center}
If $n$ is even and $\gcd(n-1,\ell)=1$, then $\overline{M_n}$ is the unique $N$-AW graph of maximum size on $n$ vertices.
\end{prop}

\begin{proof}
Each component of $M_n$ is a pendant tree. By Corollary \ref{2c-1} $\overline{M_n}$ is $N$-AW if and only if $\gcd \left( 2 \left( \frac{n}{2} \right)-1,\ell \right)=\gcd(n-1,\ell)=1$. 
That $\overline{M_n}$ is the unique $N$-AW graph of maximum size on $n$ vertices again follows from two dominating vertices being $N$-twins in any other case.
\end{proof}

It turns out that when $n$ is even finding $\max(n,\ell)$ is considerably more complicated, though we conjecture that almost all extremal graphs are complements of pendant graphs. In Proposition \ref{prop:triangle} we find the extremal graphs where $n$ is even and $\ell$ is odd. This is the only situation we have found in which an $(n,\ell)$-extremal graph is not the complement of a pendant graph.

\begin{prop}\label{prop:triangle}
Suppose that $n \ge 4$ is even.  If $\ell$ is odd  and $\gcd(n-1,\ell)\neq 1$ then $\max(n,\ell) = \binom{n}{2}-\left(\frac{n}{2}+1\right)$. In this case an $(n,\ell)$-extremal graph is $\overline{C_3 \cup \left( \frac{n-4}{2} \right) P_2 \cup K_1}$. 
\end{prop}

\begin{proof}
We first show that $H=\overline{C_3 \cup \left( \frac{n-4}{2} \right) P_2 \cup K_1}$, which has $\binom{n}{2}- \left( \frac{n}{2}+1 \right)$ edges, is $N$-AW.  By Theorem~\ref{subsetjoindom}, we need only prove that $C_3 \cup \left( \frac{n-4}{2} \right) P_2$ is $A$-AW.  Clearly, $P_2$ is $A$-AW  for all $\ell \in \mathbb{N}$. 
We can see that $C_3$ is $A$-AW if and only if $\ell$ is odd by row reducing the adjacency matrix of $C_3$.
By playing on each component, $C_3 \cup \left( \frac{n-4}{2} \right) P_2$ is $A$-AW, and so $H$ is $N$-AW. 

We now show that $H$ is $(n,\ell)$-extremal.  Since $\gcd(n-1,\ell) \ne 1$, Proposition~\ref{evenmatch} implies that $M_n$ is not $(n,\ell)$-extremal. By the uniqueness of $M_n$, 
$\max(n,\ell) \leq \binom{n}{2}-\left(\frac{n}{2}+1\right)$. Since $H$ is $N$-AW and $|E(\overline{H})|=\binom{n}{2}-\left(\frac{n}{2}+1\right)$ $\max(n,\ell)=\binom{n}{2}-\frac{n}{2}+1$.
\end{proof}

Since $\overline{M_n}$ is the $(n,\ell)$-extremal graph in the case that $n$ is odd and $\overline{C_3 \cup \left( \frac{n-4}{2} \right) P_2 \cup K_1}$ is an $(n,\ell)$-extremal graph in the case that $n$ is even and $\ell$ is odd, from here on we consider only cases where $n$ and $\ell$ are both even. In this case we find the quantity given in Conjecture \ref{conjecture} is a lower bound.

\begin{prop}\label{prop:lowerbound}
If $n$ and $\ell$ are both even then 
\[ \max(n,\ell)\geq \binom{n}{2} - \left(\frac{n}{2} + k\right)\]
where $k$ is the smallest nonnegative integer such that $\gcd(n-2k-1,\ell)=1$. 
\end{prop}
\begin{proof}
Suppose that $k$ is the smallest nonnegative integer such that $\gcd(n-2k-1,\ell)=1$. Let $G=kP_4\cup \left(\frac{n}{2} - 2k\right)P_2$. Then $G$ is a pendant graph and $\overline{G}$ has size $\binom{n}{2} - \left(\frac{n}{2} + k\right)$. So by Lemma \ref{pendantN} we have $\overline{G}$ is $N$-AW and the result follows. 
\end{proof}

Note when $k= \frac{n}{2}-1$ we have $n-2k-1 = 1$ and so $\gcd(n-2k-1,\ell)=1$. This gives us the lower bound in Proposition \ref{bestcase}. In the following two subsections we find $\max(n,\ell)$ in two cases: finding all graphs with minimum degree $n-2$ or $n-3$  that are $(n,\ell)$-extremal for any $\ell$ in Section \ref{mindegree}, and finding all combinations of $n$ and $\ell$ such that the $(n,\ell)$-extremal graph has $\binom{n}{2}-\left(\frac{n}{2}+k\right)$ edges for $0\leq k\leq 3$ in Section \ref{smallk}. 
In both perspectives we are led to pendant graphs, which supports Conjecture \ref{conjecture}.

\subsection{Extremal Graphs With a Given Minimum Degree }\label{mindegree}

The minimum degree of an $(n,\ell)$-extremal graph can not be $n-1$ because $K_n$ is not $N$-AW. Moreover, if the minimum degree is $n-2$ then, to avoid twins, the complement graph must be $M_n$. So, Propositions \ref{prop:oddmatch} and \ref{evenmatch} tell us that if $G$ is an $(n,\ell)$-extremal graph with minimum degree $n-2$, then $G=\overline{M_n}$, which is the complement of a pendant graph when $n$ is even. Thus, in this section we find $\max(n,\ell)$ among all graphs with minimum degree $n-3$. Recall we can assume $n$ and $\ell$ are even. 

If $G$ has minimum degree $n-3$ the complement has maximum degree $2$. So the components of the complement graph are paths and cycles.  We denote the cycle graph $C_k$ by $V(C_k) = \{ v_i : 1 \le i \le k \}$, $E(C_k) = \{ v_iv_{i+1}, v_kv_1 : 1 \le i \le k-1 \}$.  Our approach to the $A$-Lights Out game on $C_k$ is similar to our approach to the $N$-Lights Out game in \cite{paper11}.  We first reduce an arbitrary labeling to a canonical labeling, and then determine when these canonical labelings can be won.

To that end, for $a,b \in \mathbb{Z}_\ell$ we define  $\lambda_{a,b}$ to be the labeling where $v_1$ has label $a$, $v_2$ has label $b$, and the other vertices have label $0$.  By a straightforward induction proof, given any initial labeling of $C_k$ in the $A$-Lights Out game, the vertices can be toggled to achieve the $\lambda_{a,b}$ labeling for some $a,b \in \mathbb{Z}_\ell$.  These are our canonical labelings.  The following lemma shows how we deal with the labelings $\lambda_{a,b}$ and $(\lambda_{a,b})_s$.


\begin{lem} \label{labelingstuff}
Let $\pi$ be a labeling of $V(C_k)$, and let $\ell$ be even.

\begin{enumerate}
\item \label{winlambda} The labeling $\lambda_{a,b}$ is $A$-winnable precisely in the following circumstances.
	\begin{itemize}
	\item When $n \equiv 0$ (mod 4) and $a=b=0$.
	\item When $n \equiv 1,3$ (mod 4) and $a$ and $b$ have the same parity.
	\item When $n \equiv 2$ (mod 4) and $a$ and $b$ are both even.
	\end{itemize}
\item \label{strans} The labeling $(\lambda_{a,b})_s$ can be toggled in the $A$-Lights Out game to obtain the following labelings.
	\begin{itemize}
	\item When $n \equiv 0$ (mod 4), $\lambda_{a,b}$.
	\item When $n \equiv 1$ (mod 4), $\lambda_{a,b-s}$.
	\item When $n \equiv 2$ (mod 4), $\lambda_{a-s,b-s}$.
	\item When $n \equiv 3$ (mod 4), $\lambda_{a-s,b}$.
	\end{itemize} 
\end{enumerate} 
\end{lem}

\begin{proof}
For (\ref{winlambda}), let $t_i$ be the number of times we toggle $v_i$  By a straightforward induction proof, it follows that $\lambda_{a,b}$ is $A$-winnable if and only if $t_{n-1}+t_1=0$, $t_n+a+t_2=0$, and for $2 \le i \le n$ we have
\begin{equation*} \label{tieq}
t_i = \begin{cases} -t_2 & i \equiv 0 \hbox{ (mod 4)} \\ b+t_1 & i \equiv 1 \hbox{ (mod 4)} \\ t_2 & i \equiv 2 \hbox{ (mod 4)} \\ -b-t_1 & i \equiv 3 \hbox{ (mod 4)} \end{cases}.
\end{equation*} 
Then (\ref{winlambda}) follows from using the equations above with $i=n-1$ and $i=n$.

For (\ref{strans}), we begin with the labeling $(\lambda_{a,b})_s$, and then each $v_i$ with $2 \le i < 4\left\lfloor \frac{n}{4} \right\rfloor$ and $i \equiv 2,3$ (mod 4) is toggled $-s$ times.  This results in the labeling where each $v_i$ with  
$1 \le i \le 4  \left\lfloor \frac{n}{4} \right\rfloor$ has label $\lambda_{a,b}(v_i)$ and each $v_i$ with $i>4\left\lfloor \frac{n}{4} \right\rfloor$ has label $(\lambda_{a,b})_s(v_i)$.  This gives us the $n \equiv 0$ (mod 4) case, and the other cases follow from appropriately toggling some combination of $v_1$, $v_{n-1}$, and $v_n$.
\end{proof} 

The next result helps us see how the presence of cycle components in a graph can affect how we apply Theorem~\ref{pendantremove}(\ref{penneighbor}).

\begin{lem} \label{notswin}
Let $G$ be a graph.
\begin{enumerate}
\item \label{notswineven} If $G$ has a connected component that is a cycle of even order, then $G$ has a labeling $\pi$ such that $\pi_s$ is not $A$-winnable for all $s \in \mathbb{Z}_\ell$.
\item \label{notswinplural} If $G$ has two connected components that are cycles, then $G$ has a labeling $\pi$ such that $\pi_s$ is not $A$-winnable for all $s \in \mathbb{Z}_\ell$.
\end{enumerate}
\end{lem}

\begin{proof}
For (\ref{notswineven}), let $C$ be a cycle component of $G$ with even order, and define a labeling that is $\lambda_{1,0}$ on $C$ and arbitrary on the remaining vertices of $G$.  Since a labeling is winnable on a graph if and only if it is winnable on each connected component, it suffices to prove that $(\lambda_{1,0})_s$ is not winnable on $C$ for all $s \in \mathbb{Z}_\ell$.  If $C$ has order divisible by 4, then Lemma~\ref{labelingstuff}(\ref{strans}) implies that the vertices of $C$ can be toggled to achieve $\lambda_{1,0}$, which is not winnable by Lemma~\ref{labelingstuff}(\ref{winlambda}).  If $C$ has order not divisible by 4, then by Lemma~\ref{labelingstuff}(\ref{strans}), the vertices can be toggled to achieve the labeling $\lambda_{1-s,-s}$.  Since $1-s$ and $s$ can never both be even, Lemma~\ref{labelingstuff}(\ref{winlambda}) implies that $\lambda_{1-s,-s}$ is not winnable for all $s \in \mathbb{Z}_\ell$.  In either case, $(\lambda_{1,0})_s$ is not winnable on $C$ for all $s \in \mathbb{Z}_\ell$, and so $(\lambda_{1,0})_s$ is not $A$-winnable for all $s \in \mathbb{Z}_\ell$.

For (\ref{notswinplural}), let $C$ and $C'$ be two cycle components of $G$.  By (\ref{notswineven}), we can assume each of $C$ and $C'$ has odd order.  We claim that for any labeling $\pi$ that restricts to $\lambda_{1,0}$ on $C$ and $\lambda_{0,0}$ on $C'$, $\pi_s$ is not $A$-winnable for all $s \in \mathbb{Z}_\ell$.  By Lemma~\ref{labelingstuff}(\ref{strans}), with initial labeling $\pi_s$, we can toggle the vertices of $G$ to obtain a labeling that restricts either to $\lambda_{1-s,0}$ or $\lambda_{1,-s}$ on $C$ and restricts either to $\lambda_{-s,0}$ or $\lambda_{0,-s}$ on $C'$.  If $s$ is even, then $1-s$ and 0 as well as $1$ and $-s$ have opposite parity.  If $s$ is odd, then $-s$ and 0 have opposite parity.  In any case, Lemma~\ref{labelingstuff}(\ref{winlambda}) implies that $\pi_s$ is not winnable, and so $\pi$ is not $A$-winnable for all $s \in \mathbb{Z}_\ell$.
\end{proof} 

Our next lemma helps us when we want to apply Theorem~\ref{subsetjoindom} to graphs with both a dominating vertex and a cycle component in its complement.

\begin{lem} \label{cyclenotaw}
If $\ell$ is even, then every cycle graph is not $A$-AW.
\end{lem}

\begin{proof}
By Lemma~\ref{labelingstuff}(\ref{winlambda}), if $a,b \in \mathbb{Z}_\ell$ have opposite parity, then $\lambda_{a,b}$ is not $A$-winnable.  The result follows.
\end{proof}

The following theorem gives us a connection between $(n,\ell)$-extremal graphs and pendant graphs, in support of Conjecture \ref{conjecture}. We use $\Delta(G)$ to denote the maximum degree of a graph $G$.

\begin{thm}\label{thm:maxdegree2}
Let $\ell$ be even, and let $G$ be an $(n,\ell)$-extremal graph of even order with $\Delta(\overline{G}) \le 2$.  Then each connected component of $\overline{G}$ is either $P_2$ or $P_4$.
\end{thm} 

\begin{proof}
Suppose $\Delta(\overline{G})=1$. All dominating vertices in $G$ are $N$-twins so by Corollary \ref{twins} $G$ has at most $1$ dominating vertex.  However, $G$ can not have only one dominating vertex since $G$ has even order. Thus, $G$ has no dominating vertices, and so $\overline{G}$ has no isolated vertices.   It follows that each connected component of $G$ is $P_2$. 

In the case $\Delta(\overline{G})=2$, we first prove that $\overline{G}$ has at least one path component.  If not, all connected components are cycles, and so $|E(\overline{G})|=|V(\overline{G})|$.  However, note that any pendant tree of order $|V(G)|$ is $N$-AW for all $\ell$ by Corollary \ref{2c-1}.  Since the pendant tree has size $|V(G)|-1$, this implies that $G$ is not $(n,\ell)$-extremal.  Thus, $\overline{G}$ has at least one path component (possibly $P_1$).

By Corollary~\ref{pathrestrictions}, no component of $\overline{G}$ is $P_k$ with $k \ge 5$ or $k=3$. Furthermore two components of $P_1$ in $\overline{G}$ would be $N$-twins in $G$, which is prohibited by Corollary \ref{twins}. If we have one component of $P_1$, Theorem~\ref{subsetjoindom} implies that all other connected components of $\overline{G}$ are $A$-AW.  This excludes cycles by Lemma~\ref{cyclenotaw}.  Since the remaining paths have even order, this would force $G$ to have odd order, which is a contradiction.

So $G$ is $N$-AW and $\overline{G}$ has a pendant tree ($P_2$ or $P_4$) as a component. Thus, $\overline{G}$ has a pendant vertex, so we can use Theorem~\ref{pendantremove}(\ref{penneighbor}). 
This implies that $\overline{G}$ is $(A,\ell,s)$-winnable for some $s \in \mathbb{Z}_\ell$.  However, Lemma~\ref{notswin} implies that this can not happen if either $\overline{G}$ has more than one cycle component or if $\overline{G}$ has a cycle component of even order.  Moreover, if $\overline{G}$ has precisely one cycle component, and if that connected component has odd order, this implies that $G$ has odd order, which is a contradiction.  Thus, $\overline{G}$ has no cycle components, and so each connected component is either $P_2$ or $P_4$, which completes the proof.
\end{proof}

Note that the $(n,\ell)$-extremal graphs given in Theorem \ref{thm:maxdegree2} are pendant graphs. By Lemma \ref{pendantN} $kP_4 \cup \frac{n-4k}{2}P_2$ is $N$-AW if and only if $\gcd(n-2k-1,\ell)=1$. This implies Conjecture \ref{conjecture} for the family of graphs which have minimum degree at least $n-3$.

 \subsection{Extremal Graphs with ${n}\choose{2}$$ - (\frac{n}{2}+k)$ edges}\label{smallk}
 
In this section we prove Conjecture \ref{conjecture} for $0\leq k \leq 3$.  We state Theorem \ref{thm:extremalpendant} in the language of that conjecture.

\begin{thm}\label{thm:extremalpendant}
For $n,\ell$ even and $0\leq k \leq 3$
\[\max(n,\ell)=\binom{n}{2} - \left(\frac{n}{2} + k\right)\]
where $k$ is the smallest nonnegative integer such that $\gcd(n-2k-1,\ell)=1$. In each case the $(n,\ell)$-extremal graphs are precisely the complements of pendant graphs of order $n$ that have size $\binom{n}{2} - \left(\frac{n}{2} + k\right)$. 
\end{thm}

We will prove this result  using separate propositions for each $k$. When $k=0$, Proposition \ref{evenmatch} implies Theorem \ref{thm:extremalpendant}.  The following lemma will help us for the cases $1\leq k\leq 3$. 

\begin{lem} \label{degreebound}
Let $n \in \mathbb{N}$ be even, let $\ell\in\mathbb{N}$, and let $G$ be a $N$-AW graph with $|E(\overline{G})|=\frac{n}{2}+t$, where $t \ge 1$. Then  $\Delta(\overline{G})\leq t+1$ where $\Delta(\overline{G})$ is the maximum degree of $\overline{G}$.
\end{lem}

\begin{proof}
We let $v \in V(\overline{G})$ and show $\deg(v) \le t+1$, where $\deg(v)$ is the degree of $v$ in $\overline{G}$.  Let $W=V(\overline{G})-N_{\overline{G}}[v]$, and note that $\deg(v)=|N_{\overline{G}}(v)|$.  Then $|W|=n-\deg(v)-1$.  In the graph $\overline{G}$, let $k$ be the number of edges incident only to vertices in $N_{\overline{G}}(v)$, let $r$ be the number of edges incident only to vertices in $W$, and let $s$ be the number of edges between a vertex in $N_{\overline{G}}(v)$ and a vertex in $W$.  Since $|E(\overline{G})|=\frac{n}{2}+t$, we have $\frac{n}{2}+t = \deg(v)+k+r+s$, and so $k+r+s = \frac{n}{2}+t-\deg(v)$.

Since $G$ is $N$-AW, it can not have any $N$-twins.  Thus, no vertices in $N_{\overline{G}}(v)$ can be $N$-twins, so we can have at most one vertex in $N_{\overline{G}}(v)$ that is adjacent in $G$ to every vertex except $v$. In other words, there are at least $\deg(v)-1$ vertices in $W$ that are adjacent in $\overline{G}$ to vertices other than $v$.  There can be at most two such vertices for each of the $k$ edges in $\overline{G}$ incident with two vertices in $N_{\overline{G}}(v)$, and at most one such vertex for each of the $s$ edges between vertices in $N_{\overline{G}}(v)$ and $W$.  This means that there are at most $2k+s$ such vertices in $N_{\overline{G}}(v)$.  It follows that $\deg(v)-1 \le 2k+s$, and so $\deg(v) \leq 2k+s+1$.

In order to prevent any vertices in $W$ from becoming $N$-twins, we can have at most one vertex in $W$ that is adjacent to every vertex in $G$.  In other words, there are at least $|W|-1 = n-\deg(v)-2$ vertices in $W$ with nonzero degree in $\overline{G}$.  Similar reasoning as in the previous paragraph implies that there are at most $2r+s$ such vertices in $W$, and so $n-\deg(v)-2 \le 2r+s$.  Thus, $\deg(v) \geq n-2r-s-2$.

Since we have $n-2r-s-2 \leq \deg(v) \leq 2k+s+1$, it follows that $n-2r-s-2 \leq 2k+s+1$.  This gives us $n-2k-2r-2s \leq 3$.  Since the left side of the equation is even, this actually gives us $n-2k-2r-2s \leq 2$, and so $\frac{n}{2}-k-r-s \leq 1$.  Rearranging this a bit gives us $k+r+s \geq \frac{n}{2}-1$.

Now we use the fact $k+r+s=\frac{n}{2}+t-\deg(v)$ to get $\frac{n}{2}+t-\deg(v) \geq \frac{n}{2}-1$.  Solving for $\deg(v)$ gives deg$(v) \leq t+1$.
\end{proof}

In the next proposition, we resolve the $k=1$ case of Theorem \ref{thm:extremalpendant}.

\begin{prop}\label{prop:plus1}
Suppose that $n$ and $\ell$ are even and $n\ge 4$.  Then 
\begin{center}
$\max(n,\ell)=\binom{n}{2} - \left( \frac{n}{2}+1 \right)$ if and only if $\gcd(n-1,\ell)\neq 1$ and $\gcd(n-3,\ell)=1$.
\end{center}

\noindent Moreover, the only $(n,\ell)$-extremal graph in this case is the complement of the unique pendant graph of order $n$ and size $\binom{n}{2}-\left(\frac{n}{2} + 1\right)$, which is $\overline{P_4\cup \left( \frac{n}{2}-2 \right) P_2}$.
\end{prop}

\begin{proof}

Suppose $\gcd(n-1,\ell)\neq 1$ and $\gcd(n-3,\ell)=1$. Consider $H=P_4 \cup \left( \frac{n}{2} -2 \right) P_2$. Note that $H$ is a pendant graph with $n$ vertices and $\frac{n}{2} + 1$ edges.  By Lemma \ref{pendantN}, $\overline{H}$ is $N$-AW if and only if $\gcd(n-3,\ell)=1$. Since $\gcd(n-1,\ell)\neq 1$ it follows from Proposition \ref{evenmatch} that $\max(n,\ell)=\binom{n}{2}-\left( \frac{n}{2} + 1 \right)$.

Suppose $\max(n,\ell)=\binom{n}{2}-\left( \frac{n}{2}+1 \right)$. Then $\gcd(n-1,\ell)\neq 1$, since otherwise $\max(n,\ell) = \binom{n}{2}-\frac{n}{2}$ by Proposition \ref{evenmatch}.  By Lemma \ref{degreebound}, if $G$ is $N$-AW with $|E(\overline{G})|=\frac{n}{2}+1$  then $\Delta(\overline{G})\leq 2$. So by Theorem \ref{thm:maxdegree2} each connected component of $\overline{G}$ is either $P_2$ or $P_4$. The only such graph with $\frac{n}{2}+1$ edges is $H$. Thus $\gcd(n-3,\ell)=1$. It is clear that $H$ is the only pendant graph of order $\frac{n}{2}+1$.
\end{proof}



In the next proposition, we resolve the $k=2$ case of Theorem \ref{thm:extremalpendant}. The proof considers the possible degree sequences of the complements of $(n,\ell)$-extremal graphs. To ease our explanation we introduce a notation. Let a \emph{$d$-vertex} refer to a vertex of degree $d$. A \emph{$d^+$-vertex} is a vertex of degree $d$ or more. 

\begin{lem}\label{lem:dvertices}
Suppose $G$ is an $N$-AW graph. Then any $d$-vertex in $\overline{G}$ with $d\geq 2$ must have at least $d-1$ neighbors that are $2^+$-vertices.
\end{lem}
\begin{proof}
Suppose that $v$ is a $d$-vertex in $\overline{G}$ with $d\geq 2$ and that $v$ has fewer than $d-1$ neighbors that are $2^+$ vertices. Then $v$ has two neighbors of degree $1$ in $\overline{G}$, which results in $G$ having $N$-twins.  By Corollary \ref{twins}, $\overline{G}$ is not $N$-AW.
\end{proof}

\begin{prop}\label{prop:plus2}
Let $n,\ell\in\mathbb{N}$ be even and $n\geq 6$.  Then
\begin{center}
$\max(n,\ell)=\binom{n}{2}- \left( \frac{n}{2} + 2 \right)$ if and only if $\gcd(n-1,\ell)\neq 1$, $\gcd(n-3,\ell)\neq 1$, and $\gcd(n-5,\ell)=1$. 
\end{center}
Moreover, the $(n,\ell)$-extremal graphs in this case are precisely the complements of pendant graphs of order $n$ and size $\frac{n}{2}+2$: $\overline{(P_3 \astrosun K_1)\cup \frac{n-6}{2} P_2}$ and $\overline{2P_4\cup \frac{n-8}{2} P_2}$ with the latter only possible when $n\geq 8$.
\end{prop}
\begin{proof}
Suppose $\gcd(n-1,\ell)\neq 1$, $\gcd(n-3,\ell)\neq 1$ and $\gcd(n-5,\ell)=1$. By Proposition \ref{evenmatch} and Proposition \ref{prop:plus1} $\max(n,\ell)\leq \binom{n}{2}- \left( \frac{n}{2}+2 \right)$. Consider $H = (P_3 \astrosun K_1)\cup \frac{n-6}{2}P_2$ which has size $\frac{n}{2}+2$. By Lemma \ref{pendantN}, $\overline{H}$ is $N$-AW if and only if $\gcd(n-5,\ell)=1$. So, $\max(n,\ell)=\binom{n}{2}- \left( \frac{n}{2}+2 \right)$.

Now suppose that $\max(n,\ell)=\binom{n}{2}- \left( \frac{n}{2}+2 \right)$. Then $\gcd(n-1,\ell)\neq 1$ and $\gcd(n-3,\ell)\neq 1$ by Propositions \ref{evenmatch} and \ref{prop:plus1}. We will describe all $G$ such that $G$ is $N$-AW and $E(\overline{G})=\frac{n}{2} + 2$ and show either that these graphs are not $(n,\ell)$-extremal or that they are $N$-AW if and only if $\gcd(n-5,\ell)=1$.

Suppose that $G$ is $N$-AW with $E(\overline{G})=\frac{n}{2}+2$. The degree sum of $\overline{G}$ is $n+4$. By Lemma \ref{degreebound}, $\Delta(\overline{G}) \leq 3$. To avoid $N$-twins in $G$, $\overline{G}$ can have at most one $0$-vertex. Thus the only possible degree sequences for $\overline{G}$ are $d_0 = (3,3,1,1,\dots,1)$, $d_1=(3,2,2,1,1,\dots,1)$, $d_2 = (2,2,2,2,1,1,\dots,1)$, $d_3=(3,3,2,1,1,\dots,1,0)$, $d_4=(3,2,2,2,1,1,\dots,1,0)$, and $d_5 = (2,2,2,2,2,1,1,\dots,1,0)$.
For a graph with degree sequence $d_0$ note that each of the 3-vertices must have at least two pendant neighbors. So by  Lemma~\ref{lem:dvertices}, no graph with degree sequence $d_0$ is $N$-AW.

If $\overline{G}$ has degree sequence $d_1$, Lemma \ref{lem:dvertices} implies that all $2^+$-vertices are in the same component. The other components of $\overline{G}$ must be a matching. So our options are  $G' \cup \frac{n-4}{2} P_2$ where $G'$ is shown in Figure \ref{fig:G1} or $H=(P_3 \astrosun K_1)\cup \frac{n-6}{2}P_2$. Note $\overline{H}$ is $N$-AW if and only if $\gcd(n-5,\ell)=1$ as shown above.  The graph $G'$ has order $4$ and size $4$. Moreover, given the initial labeling $\bf{0}_1$ we can achieve the $0$ labeling by toggling the vertices $b$ and $c$ each $-1$ times. Thus, $T_{V(G_1)}^A (1) = \{-2\}$. By Lemma \ref{matrixLO}, $G'$ is $A$-AW because the adjacency matrix is invertible.  The graph $P_4 = P_2 \astrosun K_1$ has order $4$, size $3$, is $A$-AW by Lemma~\ref{pendwin}(\ref{pendwinall}), and, by Lemma~\ref{pendwin}(\ref{pendtoggen}), $T_{V(P_4)}^A(1) = \{-2\}$.  Thus, $\overline{G'\cup \frac{n-4}{2} P_2}$ is not $(n,\ell)$-extremal by Corollary~\ref{extswitch}. 

\begin{figure}
\begin{center}
\begin{tikzpicture}
\node at (0,-.75) (1) {a};
\node at (.75,0) (2) {b};
\node at (.75,-1.5) (3) {d};
\node at (1.5,0) (4) {c};
\draw (1) -- (2);
\draw (1) -- (3);
\draw (2) -- (3);
\draw (2) -- (4);
\end{tikzpicture}
\caption{One option for the non-matching component of a graph with degree sequence $d_1$ in the proof of Proposition \ref{prop:plus2}.}
\label{fig:G1}
\end{center}
\end{figure}
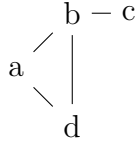

If $\overline{G}$ has degree sequence $d_2$ then $\Delta(\overline{G})=2$ and so, by Theorem \ref{thm:maxdegree2}, each component is $P_2$ or $P_4$. This leaves just $2P_4 \cup \frac{n-8}{2}P_2$ which is a pendant graph and thus, by Lemma \ref{pendantN}, $N$-AW if and only if $\gcd(n-5,\ell)=1$. 

Suppose $\overline{G}$ has degree sequence $d_3$, $d_4$ or $d_5$. In these cases $\overline{G}$ has an isolated vertex so by  Corollary~\ref{extdom}, any component with non-pendant vertices has no pendant vertices. Degree sequence $d_3$ is impossible because there are not enough $2^+$ vertices to be in a component with a $3$-vertex. If $\overline{G}$ has degree sequence $d_4$, this implies one of the components must have odd degree sum which is impossible. If $\overline{G}$ has degree sequence $d_5$ then $\Delta(\overline{G})= 2$.  By Theorem \ref{thm:maxdegree2} if $G$ is $(n,\ell)$-extremal then each component of $\overline{G}$ is either $P_2$ or $P_4$. Since the number of $2$-vertices is odd no such graph exists.

Thus if $\max(n,\ell)=\binom{n}{2} - \left( \frac{n}{2}+2 \right)$ then $\gcd(n-5,\ell)=1$. Moreover the unique $(n,\ell)$-extremal graphs are $\overline{(P_3 \astrosun K_1)\cup \frac{n-6}{2} P_2}$ and $\overline{2P_4\cup \frac{n-8}{2} P_2}$ which are the complements of the only pendant graphs of order $n$ and size $\frac{n}{2} + 2$.
\end{proof}

In the next proposition, we resolve the case $k=3$ of Theorem \ref{thm:extremalpendant}. We use the techniques of Proposition \ref{prop:plus2}, but must analyze more cases. 

\begin{prop}\label{prop:plus3}
Let $n,\ell\in\mathbb{N}$ be even and $n\geq 8$.  Then 
\begin{center}
$\max(n,\ell)=\binom{n}{2}- \left( \frac{n}{2} + 3 \right)$ if and only if $\gcd(n-2k-1,\ell)\neq 1$ for $0\leq k\leq 2$, and $\gcd(n-7,\ell)=1$.
\end{center}
In this case the unique $(n,\ell)$-extremal examples are exactly those graphs whose complements are pendant graphs with $\frac{n}{2}+3$ edges: $(C_3 \astrosun K_1) \cup \frac{n-6}{2}P_2$, $(P_4 \astrosun K_1) \cup \frac{n-8}{2} P_2$, $(K_{1,3} \astrosun K_1) \cup \frac{n-8}{2}P_2$, $(P_3 \astrosun K_1) \cup P_4 \cup \frac{n-10}{2}P_2$, and $3P_4 \cup \frac{n-12}{2} P_2$.
\end{prop}

\begin{proof}
Suppose  $\gcd(n-2k-1,\ell)\neq 1$ for $0\leq k\leq 2$, and $\gcd(n-7,\ell)=1$.  By Propositions \ref{evenmatch}, \ref{prop:plus1} and \ref{prop:plus2} we know $\max(n,\ell) \leq \binom{n}{2}- \left( \frac{n}{2}+3 \right)$. Consider $G=(P_4 \astrosun K_1)\cup \frac{n-8}{2}P_2$ which has $\frac{n}{2}+3$ edges. Since $\gcd(n-2(3)-1,\ell)=1$, $G$ is $(N,\ell)$-AW by Lemma \ref{pendantN}. So $\max(n,\ell) = \binom{n}{2}- \left( \frac{n}{2}+3 \right)$.

Suppose that $\max(n,\ell)=\binom{n}{2} - \left( \frac{n}{2}+3 \right)$. Then $\gcd(\ell,2)\neq 1$ and $\gcd(n-2k-1,\ell)\neq 1$ for $0\leq k\leq 2$ by Propositions \ref{evenmatch}, \ref{prop:plus1} and \ref{prop:plus2}.  We describe all $G$ such that $G$ is $(N,\ell)$-AW and $E(\overline{G}) = \frac{n}{2} + 3$ and show that these graphs are either not $(n,\ell)$-extremal or are $(N,\ell)$-AW if and only if $\gcd(n-7,\ell)=1$.

Suppose that $G$ is $(N,\ell)$-AW with $E(\overline{G})= \frac{n}{2}+3$. The degree sum of $\overline{G}$ is $n+6$. By Lemma \ref{degreebound}, $\Delta(\overline{G}) \leq 4$. To avoid $N$-twins in $G$, $\overline{G}$ can have at most one $0$-vertex.  

We first consider the case when $\overline{G}$ has a $0$-vertex.  By Corollary \ref{extdom} we know that $\overline{G}$ does not have a pendant vertex that is not part of a $P_2$ component. 
Thus the degree sequence of $\overline{G}$ has an even number of $1$-vertices. Since the total number of vertices is even and we have a $0$-vertex, we know the number of $2^+$-vertices will be odd. By considering all integer partitions of $7$ that when added to $(1,1,\dots,1,0)$ will satisfy having an even number of $2^+$ vertices and no $5^+$-vertex we get the following possible degree sequences:
    \begin{itemize}
    \item $d_0 = (4,4,2,1,1,\dots,1,0)$
    \item $d_1 = (4,3,3,1,1,\dots, 1,0)$
    \item $d_2 = (4,2,2,2,2,1,\dots,1,0)$
    \item $d_3 = (3,3,2,2,2,1,\dots, 1,0)$
    \item $d_4 = (2,2,2,2,2,2,2,1,\dots,1,0)$
    \end{itemize}
By Lemma \ref{lem:dvertices} degree sequences $d_0$ and $d_1$ can not have a realization that is $(N,\ell)$-AW for any $\ell$. In the case of $d_2$, by Corollary \ref{extdom} and Lemma \ref{lem:dvertices} all the $2^+$-vertices form a component with no $1$-vertices. The only realization of $(4,2,2,2,2)$ is the bowtie graph shown in Figure \ref{bowtie}. Let $H_{d_2}$ be the bowtie graph along with the required number of $P_2$ components and an isolated vertex. By Theorem \ref{subsetjoindom}, $\overline{H_{d_2}}$ is $(N,\ell)$-AW if and only if $H_{d_2}-P_1$ is $(A,\ell)$-AW. By row reducing the adjacency matrix, we find that the bowtie graph (and thus $H_{d_2}-P_1$) is $(A,\ell)$-AW if and only if $\ell$ is odd.  Thus, $\overline{H_{d_2}}$ is $(N,\ell)$-AW if and only if $\ell$ is odd, and so the graph corresponding to $d_2$ is not $(n,\ell)$-extremal by Proposition \ref{prop:triangle}.

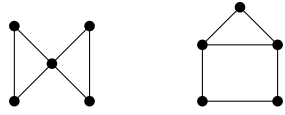
\begin{figure}
\begin{center}
\begin{tikzpicture}
\coordinate (1) at (0,0);
\coordinate (2) at (.5,.5);
\coordinate (3) at (.5,-.5);
\coordinate (4) at (-.5, .5);
\coordinate (5) at (-.5,-.5);
\fill (1) circle (2pt);
\fill (2) circle (2pt);
\fill (3) circle (2pt);
\fill (4) circle (2pt);
\fill (5) circle (2pt);
\draw (1) -- (2);
\draw (1) -- (3);
\draw (1) -- (4);
\draw (1) -- (5);
\draw (4) -- (5);
\draw (2) -- (3);

\coordinate (6) at (2,-.5);
\coordinate (7) at (3,-.5);
\coordinate (8) at (2,.25);
\coordinate (9) at (3,.25);
\coordinate (10) at (2.5, .75);
\fill (6) circle (2pt);
\fill (7) circle (2pt);
\fill (8) circle (2pt);
\fill (9) circle (2pt);
\fill (10) circle (2pt);
\draw (6) -- (7);
\draw (7)-- (9);
\draw (6) -- (8);
\draw (8) -- (9);
\draw (8) -- (10);
\draw (9) -- (10);
\end{tikzpicture}
\end{center}
\caption{The bowtie graph with degree sequence $(4,2,2,2,2)$ on the left, and the house graph  with degree sequence $(3,3,2,2,2)$ on the right. Both graphs appear in the proof of Proposition \ref{prop:plus3}.}
\label{bowtie}
\end{figure}

Again, by Corollary \ref{extdom} and Lemma \ref{lem:dvertices} we find that for the case of $d_3$, the $2^+$-vertices form a component with no $1$-vertices. Considering the cases in which the two $3$-vertices are adjacent and when they are not, we have that the only realizations of $(3,3,2,2,2)$ are the house graph (shown in Figure \ref{bowtie})  and $K_{2,3}$. Let $H_{d_3}$ be the house graph along with the required number of $P_2$ components and an isolated vertex. As in the previous paragraph, we use Theorem \ref{subsetjoindom} to determine the $N$-winnability of $\overline{H_{d_3}}$ by row reducing the adjacency matrix of the house graph, and we find $\overline{H_{d_3}}$ is never $(N,\ell)$-AW. Note that in $K_{2,3}$, the two $3$-vertices are $A$-twins and thus, by Theorem \ref{subsetjoindom}, the complement of this realization is never $(N,\ell)$-AW.

In the case of $d_4$ we note that $\Delta(\overline{G})=2$. By Theorem \ref{thm:maxdegree2} if $G$ is $(n,\ell)$-extremal then each component of $\overline{G}$ is either $P_2$ or $P_4$. Since a realization of $d_4$ has a $P_1$ component, there is no such $(n,\ell)$-extremal graph. 

Therefore, given the hypotheses, there are no $(n,\ell)$-extremal graphs with $\binom{n}{2}-\left(\frac{n}{2}+3\right)$ edges and a dominating vertex.

Now suppose there is no $0$-vertex in $\overline{G}$. To get a degree sum of $n+6$ we need to add integer partitions of $6$ to $(1,1,\dots,1)$. Considering all integer partitions of $6$ that have parts of size at most $3$  and adding these to $(1,1,\dots,1)$ we get the following possible degree sequences:
    \begin{itemize}
    \item $d_5 = (4,4,1,1,\dots,1)$
    \item $d_6 = (4,3,2,1,\dots,1)$
    \item $d_7 = (4,2,2,2,1,\dots,1)$
    \item $d_8 = (3,3,3,1,\dots,1)$
    \item $d_9 = (3,3,2,2,1,\dots,1)$
    \item $d_{10} = (3,2,2,2,2,1,\dots,1)$
    \item $d_{11}= (2,2,2,2,2,2,1,\dots,1)$
    \end{itemize}

We eliminate $d_5$ and $d_6$ using Lemma \ref{lem:dvertices}. In the case of $d_7$ all of the  $2$-vertices must be adjacent to the $4$-vertex. Considering the possible adjacencies among the  $2$-vertices the possible graphs are $H_{d_7}  =K_{1,3}\astrosun K_1\cup \frac{n-4}{2}P_2$ and the graph $G_2\cup \frac{n-6}{2} P_2$  where $G_2$ is given in Appendix \ref{Appendix}.   Since $H_{d_7}$ is a pendant graph we know $\overline{H_{d_7}}$ is $(N,\ell)$-AW if and only if $\gcd(n-7,\ell)=1$ by Lemma \ref{pendantN}. 
By Lemma \ref{matrixLO} and the fact that the adjacency matrix of $G_2$ is invertible we know $G_2$ is $(A,\ell)$-AW. From Appendix \ref{Appendix} $G_2$ has order $6$, size $6$, and $T_{G_2}^A(1) = -2$.  However, $P_3\astrosun K_1$ is $(A,\ell)$-AW by Lemma \ref{pendwin}(\ref{pendwinall}), has order $6$, size $5$, and has $T_{V(P_3\astrosun K_1)}^A =-2$ by Lemma \ref{pendwin}(\ref{pendtreetog}). Thus by Corollary \ref{extswitch}, $\overline{G_2\cup \frac{n-6}{2} P_2}$ is not $(n,\ell)$-extremal.

Consider degree sequence $d_8$. By Lemma \ref{lem:dvertices} all $3$-vertices must be adjacent to each other. Thus the only possible graph is $H_{d_8} = (C_3 \astrosun K_1)\cup \frac{n-6}{2} P_2$. Since $H_{d_8}$ is a pendant graph we know $\overline{H_{d_8}}$ is $(N,\ell)$-AW if and only if $\gcd(n-7,\ell)=1$ by Lemma \ref{pendantN}.

For degree sequence $d_9$ again by Lemma \ref{lem:dvertices} all $2^+$-vertices must be in the same component.  We generate all possible graphs with degree sequence $d_9$ by considering whether or not the two $3$-vertices are adjacent. If the two $3$-vertices are not adjacent (in the complement graph) then they each must be adjacent to both of the degree $2$ vertices, resulting in $G_{d_9}$ which is given Figure \ref{d9}. Since $\overline{G_{d_9} \cup \frac{n-6}{2} P_2}$ has $N$-twins ($v$ and $w$ in Figure \ref{d9}) this graph is not $(N,\ell)$-AW for any $\ell$ by Corollary \ref{twins}.
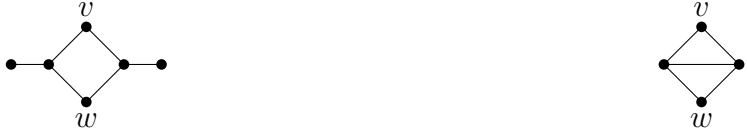
\begin{figure}
\begin{center}
\begin{multicols}{2}
\begin{tikzpicture}
\coordinate (1) at (0,0);
\coordinate (2) at (.5,.5);
\coordinate (3) at (.5,-.5);
\coordinate (4) at (1,0);
\coordinate (5) at (1.5,0);
\coordinate (6) at (-.5,0);
\fill (1) circle (2pt);
\fill (2) circle (2pt);
\fill (3) circle (2pt);
\fill (4) circle (2pt);
\fill (5) circle (2pt);
\fill (6) circle (2pt);
\draw (1) -- (2);
\draw (1) -- (3);
\draw (2) -- (4);
\draw (3) -- (4);
\draw (4) -- (5);
\draw (1) -- (6);
\draw[above] (2) node {$v$};
\draw[below] (3) node {$w$};
\end{tikzpicture}

\begin{tikzpicture}
\coordinate (1) at (0,0);
\coordinate (2) at (.5,.5);
\coordinate (3) at (.5,-.5);
\coordinate (4) at (1,0);
\fill (1) circle (2pt);
\fill (2) circle (2pt);
\fill (3) circle (2pt);
\fill (4) circle (2pt);
\draw (1) -- (2);
\draw (1) -- (3);
\draw (2) -- (4);
\draw (3) -- (4);
\draw (1) -- (4);
\draw[above] (2) node {$v$};
\draw[below] (3) node {$w$};
\end{tikzpicture}
\end{multicols}
\end{center}
\caption{At left, the graph $G_{d_9}$ - the non-matching component of the only graph with degree sequence $d_9$ from Proposition \ref{prop:plus3} for which the two $3$-vertices are not adjacent. At right, the graph $G_{d_9}'$ - the non-matching component of only graph with degree sequence $d_9$ in which the two  $3$-vertices are adjacent and have two neighbors in common in the proof of Proposition \ref{prop:plus3}.}
\label{d9}
\end{figure}

Suppose the two  $3$-vertices are adjacent. We consider cases based on their number of common neighbors. If there are no common neighbors then, to avoid twins, we get $(P_4\astrosun K_1) \cup \frac{n-8}{2} P_2$ or $G_3 \cup \frac{n-6}{2}P_2$ where $G_3$ is given in Appendix \ref{Appendix}. For the former, the complement is $(N,\ell)$-AW if and only if $\gcd(n-7,\ell)=1$ by Lemma \ref{pendantN}.  For the latter we apply Corollary \ref{extswitch}. By Lemma \ref{matrixLO}, $G_3\cup \frac{n-6}{2} P_2$ is $(A,\ell)$-AW. By Appendix \ref{Appendix} graph $G_3$ has order $6$, size $6$, and $T_{G_3}^A(1) = -2$.  However, $P_3\astrosun K_1$ is $(A,\ell)$-AW by Lemma \ref{pendwin}(\ref{pendwinall}), has order $6$, size $5$, and has $T_{V(P_3\astrosun K_1)}^A =-2$ by Lemma \ref{pendwin}(\ref{pendtreetog}). Thus by Corollary \ref{extswitch}, $\overline{G_3\cup \frac{n-6}{2} P_2}$ is not $(n,\ell)$-extremal.

Now suppose the two degree $3$ vertices have one neighbor in common. Then we get the graph  $G_4$ in Appendix \ref{Appendix}. We see $G_4$ has order $6$, size $6$, and has $T_{V(G_4)}^A(1) =\{-4\}$. Also $G_4$ is $(A,\ell)$-AW by Lemma \ref{matrixLO}. However, by Lemma \ref{pendwin}(\ref{pendtoggen}), $(P_2\cup P_1) \astrosun K_1 = P_4 \cup P_2$ has $T_{V(P_2\cup P_1)\astrosun K_1)}^A(1) = \{-4\}$.
So by Corollary \ref{extswitch} $G_4$ is not $(n,\ell)$-extremal.  Finally, suppose the two degree $3$ vertices have two neighbors in common. This results in $G_{d_9}'$ given on the right in Figure \ref{d9}. Since $\overline{G_{d_9}'\cup \frac{n-4}{2}}$ has $N$-twins ($v$ and $w$), it is not $(N,\ell)$-AW for any $\ell$ by Corollary \ref{twins}.


Next consider degree sequence $d_{10}$. In this case it is not necessarily true that all $2^+$-vertices need to be in the same component. However, if the $2^+$ vertices form more than one component, it would have to be the case that one component had a $3$-vertex and two $2$-vertices by Lemma \ref{lem:dvertices}. In this case, the arguments in Proposition \ref{prop:plus2} for degree sequence $d_1$ apply. 

Now suppose all the $2^+$-vertices are in the same component. We consider cases based on the degrees of the vertices adjacent to the $3$-vertex. This could either be two $2$-vertices and one $1$-vertex or three $2$-vertices.

Suppose there are two $2$-vertices and one $1$-vertex adjacent to the  $3$-vertex. The two  $2$-vertices each have an additional neighbor (not the $3$-vertex and not each other since then the $3$-vertex would be forced to have three neighbors of degree $2$). Call these additional neighbors $v$ and $w$. If one of these vertices is degree $1$ we end up with the graph $G_5$ in Appendix \ref{Appendix}. By Lemma \ref{pendwin}(\ref{pendtoggen}),  $P_4 \cup P_4$ has $T_{V(P_4\cup P_4)}^A =-4$, and we find $G_5 \cup \frac{n-8}{2}P_2$ is not $(n,\ell)$-extremal by Corollary \ref{extswitch} .  If $v$ and $w$ each have degree $2$ we consider the possibility that they are adjacent to each other and if they are not. This yields the graphs in Figure \ref{d9-2} and graph $G_6$ in Appendix \ref{Appendix}. The graph in Figure \ref{d9-2} has a labeling that is not $(A,\ell,s)$-winnable for all $s$, namely the labeling where one of the pendant vertices adjacent to a 2-vertex has label 1 and the remaining vertices have label 0.
 This would make $G$ not $(N,\ell)$-AW by Theorem~\ref{pendantremove}(\ref{penneighbor}). Since  $P_3 \astrosun K_1$ has $T_{V(P_3\astrosun K_1)}^A =-2$ by Lemma \ref{pendwin}(\ref{pendtreetog}), $G_6 \cup \frac{n-6}{2}P_2$ is not $(n,\ell)$-extremal by Corollary \ref{extswitch}.

Now suppose there are three $2$-vertices adjacent to the $3$-vertex. Considering whether two of the $2$-vertices are adjacent to each other or not we get graphs $G_7$ and $G_8$ in Appendix \ref{Appendix}. Since $T_{V(P_4 \cup 2P_2)}^A(1) = -6$ and $T_{V(P_4\cup P_4)}^A(1)=-4$ by Lemma \ref{pendwin}(\ref{pendtoggen}), we know neither $G_7 \cup \frac{n-6}{2}P_2$ nor $G_8 \cup \frac{n-8}{2}P_2$ is $(n,\ell)$-extremal by Corollary \ref{extswitch}.

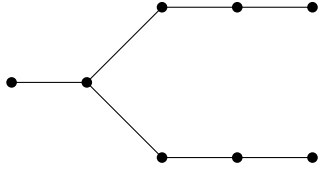
\begin{figure}
\begin{center}
\begin{tikzpicture}
\coordinate (1) at (0,0);
\coordinate (2) at (1,0);
\coordinate (3) at (2,1);
\coordinate (4) at (2,-1);
\coordinate (5) at (3,1);
\coordinate (6) at (3,-1);
\coordinate (7) at (4,-1);
\coordinate (8) at (4,1);

\fill (1) circle (2pt);
\fill (2) circle (2pt);
\fill (3) circle (2pt);
\fill (4) circle (2pt);
\fill (5) circle (2pt);
\fill (6) circle (2pt);
\fill (7) circle (2pt);
\fill (8) circle (2pt);

\draw (1) -- (2);
\draw (2) -- (3);
\draw (2) -- (4);
\draw (3) -- (5);
\draw (4) -- (6);
\draw (6) -- (7);
\draw (5) -- (8);
\end{tikzpicture}
\end{center}
\caption{The second possibility for a high degree component for degree sequence $d_{10}$ in which the degree $3$ vertex is adjacent to two degree $2$ vertices and one degree $1$ component.}
\label{d9-2}
\end{figure}

If $\overline{G}$ has degree sequence $d_{11}$ then $\Delta(\overline{G})=2$. By Theorem \ref{thm:maxdegree2} if $G$ is $(n,\ell)$-extremal then each component of $\overline{G}$ is either $P_2$ or $P_4$. Thus, the graph must be $3P_4 \cup \frac{n-12}{2} P_2$, a pendant graph. This is $(N,\ell)$-AW if and only if $\gcd(n-7,\ell)=1$ by Lemma \ref{pendantN}. 

Therefore every possible graph that is $(N,\ell)$-AW and has $E(\overline{G}) = \frac{n}{2}+3$ is either not $(n,\ell)$-extremal or has $\gcd(n-7,\ell)=1$, as desired.
\end{proof}

\noindent\emph{Proof of Theorem \ref{thm:extremalpendant}}. The cases $0\leq k\leq 3$ are true by Propositions \ref{evenmatch}, \ref{prop:plus1}, \ref{prop:plus2}, \ref{prop:plus3}, respectively. \hfill $\qed$

\section{Open Problems}

We close with three open problems related to our results.

\bigskip

\noindent (1) \emph{Does Theorem~\ref{thm:extremalpendant} hold for $k \ge 4$?} We made much progress on this result by considering the possible degree sequences.  However, when $k=4$, there are $37$  partitions of $7$ and $8$. Even with the additional restriction of Lemma \ref{degreebound} there are $23$ different degree sequences to consider. Thus, we need an alternative method to solve the general problem.

\bigskip

\noindent (2) \emph{What are the graphs of maximum size that are $(N,\ell)$-AW for all $\ell$?}  The best candidates we have found are complements of pendant trees, which have size $\bc{n}{2}-(n-1)$.  They are all $(N,\ell)$-AW for all $\ell$, but it is not clear that they are $(n,\ell)$-extremal.  

\bigskip

\noindent (3) \emph{What are the $(n,\ell)$-extremal graphs for other Lights Out games, such as the adjacency game?}

\bibliographystyle{amsalpha}
\bibliography{lightsout}

\newpage

\appendix

\section{Replacement Graphs}
\label{Appendix}

\small{The following graphs are the connected components we replace when we apply Corollary~\ref{extswitch}.  For each graph $G$, we show a picture of the graph; a table showing how many times each vertex is toggled to win the $(A,\ell)$-Lights Out game with initial labeling $\textbf{0}_1$; and $T_{V(G)}^A(1)$, which is obtained by adding the toggles from the table.  Note that in each case $A(G)$ is invertible. Thus, there is precisely one way to win this Lights Out game, which is why there is only one toggling number for each graph.}

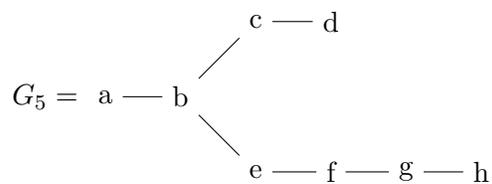
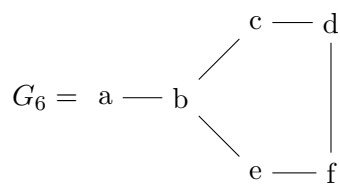
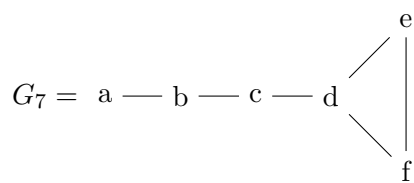
\begin{tabular}{l l l}

\adjustbox{valign=t}{\begin{tikzpicture}
\node at (-.8,-.75) (0) {$G_1=$};
\node at (0,-.75) (1) {a};
\node at (.75,0) (2) {b};
\node at (.75,-1.5) (3) {d};
\node at (1.5,0) (4) {c};

\draw (1) -- (2);
\draw (1) -- (3);
\draw (2) -- (3);
\draw (2) -- (4);

\end{tikzpicture} }&

\adjustbox{valign=t}{\begin{tabular}{|c|c|} \hline
vertex & Number of Toggles \\ \hline
a & 0  \\ \hline
b & -1 \\ \hline
c & -1  \\ \hline
d & 0  \\ \hline
\end{tabular}}

& $T_{V(G_1)}^A(1)=-2$ \\
\\

\hline
\hline

\\
\\
\\
\\
\adjustbox{valign=t}{\begin{tikzpicture}
\node at (-.8,0) (0) {$G_2=$};
\node at (0,0) (1) {a};
\node at (.75,.75) (2) {b};
\node at (.75,-.75) (3) {e};
\node at (1.5,.75) (4) {c};
\node at (2.25,.75) (5) {d};
\node at (1.5,-.75) (6) {f};

\draw (1) -- (2);
\draw (1) -- (3);
\draw (2) -- (3);
\draw (2) -- (4);
\draw (4) -- (5);
\draw (2) -- (6);
\end{tikzpicture} }&

\adjustbox{valign=t}{\begin{tabular}{|c|c|} \hline
vertex & Number of Toggles \\ \hline
a & 0  \\ \hline
b & -1 \\ \hline
c & -1  \\ \hline
d & 0  \\ \hline
e & 0  \\ \hline
f & 0  \\ \hline
\end{tabular}}

& $T_{V(G_2)}^A(1)=-2$ \\

\\

\hline
\hline

\\
\\

\adjustbox{valign=t}{\begin{tikzpicture}
\node at (-1.5,-.5) (0) {$G_3=$};
\node at (0,0) (1) {b};
\node at (1.5,0) (2) {c};
\node at (0,-1) (3) {e};
\node at (1.5,-1) (4) {f};
\node at (2.25,0) (5) {d};
\node at (-.75,0) (6) {a};

\draw (1) -- (2);
\draw (1) -- (3);
\draw (2) -- (4);
\draw (3) -- (4);
\draw (2) -- (5);
\draw (1) -- (6);
\end{tikzpicture}} &

\adjustbox{valign=t}{\begin{tabular}{|c|c|} \hline
vertex & Number of Toggles \\ \hline
a & 0  \\ \hline
b & -1 \\ \hline
c &  -1 \\ \hline
d & 0  \\ \hline
e & 0  \\ \hline
f & 0 \\ \hline
\end{tabular}}

& $T_{V(G_3)}^A(1)=-2$\\
\\

\hline
\hline

\\
\\

\adjustbox{valign=t}{\begin{tikzpicture}
\node at (-.8,0) (0) {$G_4=$};
\node at (0,0) (1) {a};
\node at (.75,.75) (2) {b};
\node at (.75,-.75) (3) {e};
\node at (1.5,.75) (4) {c};
\node at (2.25,.75) (5) {d};
\node at (1.5,-.75) (6) {f};

\draw (1) -- (2);
\draw (1) -- (3);
\draw (2) -- (3);
\draw (2) -- (4);
\draw (4) -- (5);
\draw (3) -- (6);
\end{tikzpicture}} &

\adjustbox{valign=t}{\begin{tabular}{|c|c|} \hline
vertex & Number of Toggles \\ \hline
a & 1  \\ \hline
b & 0 \\ \hline
c & -1  \\ \hline
d & -1  \\ \hline
e & -1  \\ \hline
f & -2  \\ \hline
\end{tabular}}

& \adjustbox{valign=t}{$T_{V(G_4)}^A(1)=-4$}

\end{tabular}

\begin{tabular}{ l l l }

\adjustbox{valign=t}{\begin{tikzpicture}
\node at (-.8,0) (0) {$G_5=$};
\node at (0,0) (1) {a};
\node at (1,0) (2) {b};
\node at (2,1) (3) {c};
\node at (2,-1) (4) {e};
\node at (3,1) (5) {d};
\node at (3,-1) (6) {f};
\node at (4,-1) (7) {g};
\node at (5,-1) (8) {h};

\draw (1) -- (2);
\draw (2) -- (3);
\draw (2) -- (4);
\draw (3) -- (5);
\draw (4) -- (6);
\draw (6) -- (7);
\draw (7) -- (8);
\end{tikzpicture} }&

\adjustbox{valign=t}{\begin{tabular}{|c|c|} \hline
vertex & Number of Toggles \\ \hline
a & 0  \\ \hline
b & -1 \\ \hline
c & -1  \\ \hline
d & 0  \\ \hline
e & 0 \\ \hline
f & 0 \\ \hline
g & -1 \\ \hline
h & -1 \\ \hline
\end{tabular}}

& $T_{V(G_5)}^A(1)=-4$

\\
\\

\hline
\hline

\\
\\

\adjustbox{valign=t}{\begin{tikzpicture}
\node at (-.8,0) (0) {$G_6=$};
\node (0,0) (1) {a};
\node at (1,0) (2) {b};
\node at (2,1) (3) {c};
\node at (2,-1) (4) {e};
\node at (3,1) (5) {d};
\node at (3,-1) (6) {f};

\draw (1) -- (2);
\draw (2) -- (3);
\draw (2) -- (4);
\draw (3) -- (5);
\draw (4) -- (6);
\draw (6) -- (5);
\end{tikzpicture}} &

\adjustbox{valign=t}{\begin{tabular}{|c|c|} \hline
vertex & Number of Toggles \\ \hline
a & 1  \\ \hline
b & -1 \\ \hline
c & -1  \\ \hline
d & 0  \\ \hline
e & -1 \\ \hline
f & 0 \\ \hline
\end{tabular} }

& $T_{V(G_6)}^A(1)=-2$  \\

\\

\hline
\hline

\\
\\

\adjustbox{valign=t}{\begin{tikzpicture}
\node at (-2.8,0) (0) {$G_7=$};
\node at (0,0) (1) {c};
\node at (1,0) (2) {d};
\node at (2,1) (3) {e};
\node at (2,-1) (4) {f};
\node at (-1,0) (5) {b};
\node at (-2,0) (6) {a};

\draw (1) -- (2);
\draw (2) -- (3);
\draw (2) -- (4);
\draw (1) -- (5);
\draw (3) -- (4);
\draw (6) -- (5);
\end{tikzpicture}} &

\adjustbox{valign=t}{\begin{tabular}{|c|c|} \hline
vertex & Number of Toggles \\ \hline
a & -2  \\ \hline
b & -1 \\ \hline
c & 1  \\ \hline
d &  0 \\ \hline
e & -1 \\ \hline
f & -1 \\ \hline
\end{tabular} }

& $T_{V(G_7)}^A(1)=-4$ \\
\\
\hline
\hline
\\
\\

\adjustbox{valign=t}{\begin{tikzpicture}
\node at (-2.8,0) (0) {$G_8=$};
\node at (0,0) (1) {c};
\node at (1,0) (2) {d};
\node at (2,1) (3) {e};
\node at (2,-1) (4) {g};
\node at (3,1) (5) {f};
\node at (3,-1) (6) {h};
\node at (-1,0) (7) {b};
\node at (-2,0) (8) {a};

\draw (1) -- (2);
\draw (2) -- (3);
\draw (2) -- (4);
\draw (3) -- (5);
\draw (4) -- (6);
\draw (1) -- (7);
\draw (7) -- (8);
\end{tikzpicture}} &

\adjustbox{valign=t}{\begin{tabular}{|c|c|} \hline
vertex & Number of Toggles \\ \hline
a &  -2 \\ \hline
b & -1 \\ \hline
c & 1  \\ \hline
d & 0  \\ \hline
e &-1  \\ \hline
f & -1 \\ \hline
g &-1  \\ \hline
h & -1 \\ \hline
\end{tabular}}

& $T_{V(G_8)}^A(1)=-6$

\end{tabular}

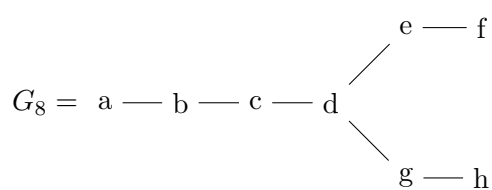
\end{document}